\documentclass[pdflatex,sn-mathphys-num]{sn-jnl}% Math and Physical Sciences Numbered Reference Style
\usepackage{graphicx}%
\usepackage{multirow}%

\usepackage{amsthm}%
\usepackage{mathrsfs}%
\usepackage[title]{appendix}%
\usepackage{xcolor}%
\usepackage{textcomp}%
\usepackage{manyfoot}%
\usepackage{booktabs}%
\usepackage{algorithm}%
\usepackage{algorithmicx}%
\usepackage{algpseudocode}%
\usepackage{listings}%
\usepackage{verbatim}
\usepackage{graphicx}
\usepackage{rotating}

\usepackage{adjustbox}%for rescaling tikzcd diagrams

\usepackage{tikz}
\usepackage{tikz-cd}
\usetikzlibrary{graphs,positioning}
\usetikzlibrary{backgrounds,circuits,circuits.ee.IEC,shapes,fit,matrix}

\tikzstyle{simple}=[-,line width=2.000]
\tikzstyle{arrow}=[-,postaction={decorate},decoration={markings,mark=at position .5 with {\arrow{>}}},line width=1.100]
\pgfdeclarelayer{edgelayer}
\pgfdeclarelayer{nodelayer}
\pgfsetlayers{edgelayer,nodelayer,main}

\tikzstyle{none}=[inner sep=0pt]

\tikzstyle{species}=[circle,fill=yellow,draw=black,scale=1.15]
\tikzstyle{transition}=[rectangle,fill=lblue,draw=black,scale=1.15]
\tikzstyle{inarrow}=[->, >=stealth, shorten >=.03cm,line width=1.5]
\tikzstyle{empty}=[circle,fill=none, draw=none]
\tikzstyle{inputdot}=[circle,fill=purple,draw=purple, scale=.25]
\tikzstyle{inputarrow}=[->,draw=purple, shorten >=.05cm]
\tikzstyle{simple}=[-,draw=purple,line width=1.000]

\usepackage{url}

\usepackage{xspace}
\usepackage{pslatex}

\usepackage{datetime}
\newdateformat{versiondate}{%
\THEMONTH\THEDAY}

\usepackage{amsthm}
\usepackage{amsmath} % flush equations left
\usepackage{amsfonts}
\usepackage{mathrsfs} 
\usepackage{amssymb}

\usepackage{mathtools}
\DeclarePairedDelimiter{\abs}{\lvert}{\rvert}
\usepackage[inline]{enumitem}
\usepackage{framed} %for framing of computational problems 

\usepackage{tabularx,booktabs}
\usepackage{xltabular}
\usepackage{subcaption}

\usepackage{todonotes}
\usepackage{environ}
\usetikzlibrary{quiver}
\usetikzlibrary{arrows}
\usetikzlibrary{cd}
\usetikzlibrary{fadings}
\usetikzlibrary{decorations.pathreplacing}
\usetikzlibrary{decorations.pathmorphing}
\tikzset{snake it/.style={decorate, decoration=snake}}

\usetikzlibrary{calc}

\newsavebox{\pullbacksquare}
\savebox{\pullbacksquare}{%
  \begin{tikzpicture}[scale=0.15]
    \draw (0,0) -- (1,0) -- (1,1);
  \end{tikzpicture}%
}
\theoremstyle{thmstyleone}
\newtheorem{theorem}{Theorem}[section]
\newtheorem{lemma}[theorem]{Lemma}
\newcommand{\define}[1]{\textbf{#1}}

\newcommand{\cat}[1]{\mathsf{#1}}
\newcommand{\typesetcategory}[1]{\cat{#1}}

\newcommand{\typesetoperator}[1]{\mathsf{#1}}

\DeclareMathOperator{\id}{\typesetoperator{id}}

\DeclareMathOperator{\Set}{\typesetcategory{Set}}

\DeclareMathOperator{\dom}{\typesetoperator{dom}}
\DeclareMathOperator{\cod}{\typesetoperator{cod}}
\DeclareMathOperator{\T}{\typesetcategory{T}}
\DeclareMathOperator{\D}{\typesetcategory{D}}

\newcommand{\cotw}[1]{#1^{\scriptstyle\triangleright}_{\scriptstyle\triangleleft}}
\newcommand{\cotwdom}{\cotw{\dom}}
\newcommand{\cotwcod}{\cotw{\cod}}
\newcommand{\cotwK}{\cotw{\K}}
\newcommand{\cotwP}{\cotw{\fP}}

\DeclareMathOperator{\cotwD}{\cotw{\cat{D}}}

\DeclareMathOperator{\K}{\mathscr{K}}
\DeclareMathOperator{\fP}{\mathscr{P}}

\DeclareMathOperator{\Pe}{\typesetcategory{Pe(T,D)}}
\DeclareMathOperator{\Cu}{\typesetcategory{Cu(T,D)}}
\DeclareMathOperator{\Nar}{\typesetcategory{Nar(T,\cotwD)}}
\DeclareMathOperator{\Vect}{\typesetcategory{Vect}}
\DeclareMathOperator{\CellSh}{\typesetcategory{CellSh}}
\newcommand{\inc}[1]{\typesetcategory{inc(#1)}}
\newcommand{\incG}{\inc{G}}
\newcommand{\incH}{\inc{H}}
\DeclareMathOperator*{\colim}{colim}

\usepackage{overpic}
\newcommand{\N}{\mathbb{N}}

\newcommand{\cG}{\mathcal{G}}

\newcommand{\Grph}{\textup{\textsf{Grph}}}

\newcommand{\Ob}{\textup{Ob}}

\newcommand{\op}{\textup{op}}

\usepackage{amsmath,amssymb,amsfonts}%

\newtheorem{proposition}[theorem]{Proposition}% 

\theoremstyle{thmstyletwo}%
\newtheorem{example}{Example}%
\newtheorem{remark}{Remark}%

\theoremstyle{thmstylethree}%
\newtheorem{definition}{Definition}%

\begin{document}

\title[Narratives for Time-Varying Data]{Time-Varying Data as Sheaves: an Invitation to Narratives}

\author*[1]{\fnm{Wilmer} \sur{Leal}}\email{wleal@ufl.edu}

\author[2]{\fnm{Benjamin Merlin} \sur{Bumpus}}\email{bumpus@usp.br}

\author[3]{\fnm{Jana K.} \sur{Nickel}}\email{jana@nickel-math.com}

\author[4]{\fnm{Johan} \sur{García}}\email{jfgarciava@unal.com}

\author[1]{\fnm{James} \sur{Fairbanks}}\email{fairbanksj@ufl.edu}

\author[1,5]{\fnm{Warren} \sur{Dixon}}\email{wdixon@\{ufl.edu,vt.edu\}}

\affil*[1]{%
  \orgdiv{Department of Mechanical and Aerospace Engineering},
  \orgname{University of Florida},
  \orgaddress{%
    \street{1064 Center Drive},
    \city{Gainesville},
    \postcode{32611-6250},
    \state{Florida},
    \country{USA}
  }
}

\affil[2]{\orgdiv{Instituto de Matemática e Estatística}, \orgname{Universidade de São Paulo}, \orgaddress{\street{Rua do Matão, 1010}, \city{São Paulo}, \postcode{05508--090}, \state{SP}, \country{Brasil}}}

\affil[3]{\orgdiv{Fachbereich Mathematik},
\orgname{Universität Hamburg},
\orgaddress{\street{Bundesstraße 55},
\city{Hamburg},
\postcode{20146},
\country{Germany}}}

\affil[4]{\orgdiv{Departamento de Matemáticas}, \orgname{Universidad Nacional de Colombia -- sede Medellín}, \orgaddress{\street{Calle 59A No. 63-20}, \city{Medellín}, \country{Colombia}}}

\affil[5]{%
\orgname{Virginia Tech},
\orgdiv{College of Engineering},
\orgaddress{
\city{Blacksburg},
\state{Virginia},
\postcode{24061},
\country{USA}
}}

\abstract{
    Modern science and engineering increasingly rely on time-varying data, yet the mathematical tools used to model temporal phenomena are often developed within separate disciplines, obscuring common principles and limiting the transfer of ideas across fields. This chapter presents the theory of narratives, an abstract framework for time-varying objects of any mathematical kind that supports both theoretical investigations and applications. To illustrate this perspective, the chapter develops three vignettes, each illustrating a different research direction. The first addresses a general concern: \textit{What information loss can occur when switching between different representations of temporal data?} The second concerns structural and algorithmic approaches: \textit{How can we systematically decompose time-varying data into simple pieces and obtain invariants describing its structural complexity?} The third is an application to control theory: \textit{How can we model multi-agent systems with switching communication topologies?} More important than any individual vignette, the central message of this invitation is that a suitable abstract perspective can organize and guide research across remarkably diverse mathematical and scientific domains.

}

\keywords{Time-Varying Data, Temporal Networks, Categories of Narratives, Temporal Data Structures, Persistent and Cumulative Data, Structured Decompositions, Cellular Sheaves, Multi-Agent Systems, Control Theory}


\pacs[MSC Classification]{
18A40, % Adjoint functors, limits and colimits
18F20, % Categories of sheaves
18A25, % Diagram schemes, limits and colimits (including directed colimits)
05C90, % Applications of graph theory
37B55, % Applications of topological and symbolic dynamics (broad dynamical systems)
93C85  % Automated systems and control (networks, distributed systems, etc.)
}

\maketitle

\tableofcontents
\newpage

\section{Introduction}

Modern science and engineering are increasingly driven by the analysis of time-varying data. From social dynamics~\cite{miritello2013temporal}, environmental science~\cite{CHOI2021105189}, and public health~\cite{MELIKER20111,niu2026temporal} to distributed control of multi-agent systems~\cite{MesbahiEgerstedt2010,OlfatiSaber2007}, language evolution~\cite{Teich2025}, and the computational history of science~\cite{Laubichler2013,llanos19}, observations rarely describe static objects; rather, they capture the evolution of phenomena through time~\cite{Atluri2018}. Despite the ubiquity of such data, the mathematical tools used to analyze them remain largely fragmented~\cite{LaxmanSastry2006}, with different communities developing specialized theories for temporal data mining, temporal databases, temporal networks, and related domains~\cite{habereder2025,LaxmanSastry2006,RoddickSpiliopoulou2002,RODDICK1992249}. Many existing approaches represent temporal information as sequences of measurements or snapshots of evolving systems, often without a coherent mathematical structure describing how these observations relate across time. As a result, integrating different representations of temporal data, or reasoning systematically about their transformations, remains a significant challenge. The theory of narratives introduced in~\cite{theorytimevaryingdata2026} addresses this problem by representing time-varying data as sheaves and cosheaves over categories of time intervals. A persistent narrative records information that remains valid throughout an interval, while a cumulative narrative records information that accumulates over it. Because narratives are categorical objects, the framework describes not only time-varying data but also maps between them, and it applies to data valued in categories of sets, graphs, topological spaces, cellular sheaves, and many other structured objects. This chapter is intended as an invitation to the theory of narratives. Through a survey of recent developments, we hope to illustrate how narratives provide a common framework in which ideas from many areas of mathematics can contribute to the study of time-varying data, through both abstract theory and concrete applications.

The theory of narratives was designed to address five requirements that arise when one seeks a general mathematical language for time-varying data:

\begin{enumerate}[label=\textbf{(D\arabic*)}]
\item (\textbf{Categories of Temporal Data}) Any theory of temporal data should define not only time-varying data, but also appropriate morphisms thereof.\label{desideratum-1}
\item (\textbf{Cumulative and Persistent Perspectives}) In contrast to being a mere sequence, temporal data should explicitly record whether it is to be viewed cumulatively or persistently. Furthermore, there should be methods of conversion between these two viewpoints.\label{desideratum-2}
\item (\textbf{Systematic ``Temporalization''}) Any theory of temporal data should come equipped with systematic ways of obtaining temporal analogues of notions relating to static data.\label{desideratum-3}
\item (\textbf{Object Agnosticism}) Theories of temporal data should be object agnostic and applicable to any kinds of data originating from given underlying dynamics.\label{desideratum-4}
\item (\textbf{Sampling}) Since temporal data arises from an underlying dynamical system, any theory of temporal data should be interoperable with theories of dynamical systems.\label{desideratum-5}
\end{enumerate}

This chapter collects three recent developments in the theory of narratives. Each is presented as a vignette centered on a different question: (1) how much information is preserved when one changes between persistent and cumulative viewpoints, (2) how the structural complexity of time-varying data can be measured, and (3) how narratives can be used to model multi-agent systems whose interaction structure changes through time. Together, these examples illustrate the broader goals of our research program:  by studying time-varying data systematically and abstractly, we can use a unified language to describe many different kinds of questions concerning time-varying data.

The first vignette concerns the adjunction between persistent and cumulative narratives. Since this adjunction is not, in general, an equivalence, converting a narrative from one viewpoint to the other and back may lose information. We introduce narratives valued in a cotwisted arrow category, which encode a persistent component, a cumulative component, and a comparison between them. We then use this category to factorize the persistence--accumulation adjunction and to classify narratives according to whether they can be recovered after either of the two possible round trips.

The second vignette concerns decompositions of time-varying data. Complex objects are often studied by expressing them as composites of smaller pieces and recording how these pieces overlap. Such decompositions also give rise to measures of structural complexity. In the temporal setting, the pieces must themselves vary through time: decomposing each snapshot independently does not record how the pieces persist, merge, split, or disappear. Building on recent work~\cite{decomposingtimevaryingdata2026}, we explain how theories of structured decompositions for static objects can be lifted to persistent narratives. The resulting framework yields decompositions whose pieces are time-varying objects and yields temporal analogues of structural invariants such as tree-width.

The third vignette applies narratives to multi-agent systems whose communication topology changes through time. By taking cellular sheaves as the category of values, a narrative can record both the changing interaction graph and the local state spaces, sensing maps, and compatibility constraints carried by its vertices and edges. The resulting model of a switching system contains substantially more information than a sequence of communication graphs. It also provides a setting in which the relationship between an underlying dynamical system and the temporal data sampled from it can be studied using the same categorical language.

\section{Background on Narratives: How to Model Time-Varying Data using Sheaves}\label{sec:foundations}

\subsection{Categories of temporal data}
We begin by recalling background on \textit{narratives}: objects introduced in~\cite{theorytimevaryingdata2026} representing time-varying objects valued in any sufficiently nice category. Narratives are (co)sheaves over categories of time intervals---we call these \textit{time categories}---which model temporal data by recording how temporal information relates across overlapping intervals of time.

The thesis of \cite{theorytimevaryingdata2026} is that temporal data should be understood not merely as a sequence of observations, but as a structured system of relations connecting observations across time. For example, a time-varying graph may consist of a family of graphs $G_t$ indexed by time together with morphisms describing how vertices and edges persist, merge, or disappear as time evolves. Narratives capture precisely these additional relations across time. Rather than representing time merely as a totally ordered set, the framework uses categories of intervals, which naturally encode the inclusion relations between time periods.

\begin{definition}[Time Categories]
A \define{time category} is any sub-join-semilattice of $\cat{I}$ or $\cat{I}_\mathbb{N}$, where $\cat{I}$ (resp. $\cat{I}_\mathbb{N}$) is the category of closed intervals in $\mathbb{R}$ (resp. $\mathbb{N}$) with inclusions as morphisms.
\end{definition}
\begin{remark}[Alternative models of time]
    Although we focus on interval categories in this chapter, the categorical viewpoint naturally accommodates more general notions of time. For instance, one could replace intervals by a finite branching tree to model branching temporal evolutions, such as the history of a Git repository or the multiple futures envisioned in Borges' \emph{The Garden of Forking Paths}~\cite{Borges1941}. Developing such generalized notions of time is an interesting direction for future research.
\end{remark}

\begin{example}[Finite time categories]\label{ex:finite-time-categories} If \(\cat T\) is a time category and \([a,b]\in\cat T\), then the slice category \(\cat T/[a,b]\), consisting of the intervals contained in \([a,b]\), is itself a time category, and the canonical inclusion \(\cat T/[a,b]\hookrightarrow\cat T\) restricts the temporal domain to the interval \([a,b]\). In particular, for every \(n\ge0\), we write \(\cat T_n:=\cat I_{\mathbb N}/[0,n]\). Its objects are the intervals \([i,j]\subseteq[0,n]\), ordered by inclusion, so \(\cat T_n\) models temporal data consisting of \(n+1\) snapshots together with every interval determined by them. For example, \(\cat T_2\) consists of the six intervals contained in \([0,2]\): 
% https://q.uiver.app/#q=WzAsNixbMCwyLCJbMCwwXSJdLFsyLDIsIlsxLDFdIl0sWzQsMiwiWzIsMl0iXSxbMSwxLCJbMCwxXSJdLFszLDEsIlsxLDJdIl0sWzIsMCwiWzAsMl0iXSxbMCwzLCIiLDAseyJzdHlsZSI6eyJ0YWlsIjp7Im5hbWUiOiJob29rIiwic2lkZSI6InRvcCJ9fX1dLFsxLDQsIiIsMCx7InN0eWxlIjp7InRhaWwiOnsibmFtZSI6Imhvb2siLCJzaWRlIjoidG9wIn19fV0sWzIsNCwiIiwyLHsic3R5bGUiOnsidGFpbCI6eyJuYW1lIjoiaG9vayIsInNpZGUiOiJib3R0b20ifX19XSxbMyw1LCIiLDIseyJzdHlsZSI6eyJ0YWlsIjp7Im5hbWUiOiJob29rIiwic2lkZSI6InRvcCJ9fX1dLFs0LDUsIiIsMix7InN0eWxlIjp7InRhaWwiOnsibmFtZSI6Imhvb2siLCJzaWRlIjoiYm90dG9tIn19fV0sWzEsMywiIiwyLHsic3R5bGUiOnsidGFpbCI6eyJuYW1lIjoiaG9vayIsInNpZGUiOiJib3R0b20ifX19XV0=
\[\begin{tikzcd}
	&& {[0,2]} && \\
	& {[0,1]} && {[1,2]} \\
	{[0,0]} && {[1,1]} && {[2,2]}
	\arrow[hook, from=2-2, to=1-3]
	\arrow[hook', from=2-4, to=1-3]
	\arrow[hook, from=3-1, to=2-2]
	\arrow[hook', from=3-3, to=2-2]
	\arrow[hook, from=3-3, to=2-4]
	\arrow[hook', from=3-5, to=2-4]
\end{tikzcd}\]
\end{example}

\begin{example}[Changing temporal resolution]\label{ex:changing-temporal-resolution}
Suppose that \([a,b]\) is an interval in a time category \(\cat T\), and choose a partition \(a=r_0\le r_1\le\cdots\le r_n=b\). This partition determines a unique functor
\[
R:\cat T_n\longrightarrow\cat T/[a,b],
\qquad
R([i,j])=[r_i,r_j].
\]
The finite time category \(\cat T_n\) models the temporal resolution of \([a,b]\) determined by the chosen snapshots \(r_0,\ldots,r_n\): the objects \([i,i]\) correspond to the snapshots, while the objects \([i,j]\) correspond to the intervals from \(r_i\) to \(r_j\).
\end{example}

To speak of sheaves, one must first have a \textit{site}, that is, a category equipped with a \textit{coverage}, where, intuitively, a coverage amounts to a systematic way of defining what it means for an object to be ``covered'' by opens. As shown in \cite{theorytimevaryingdata2026}, time categories become sites when equipped with the Johnstone coverage~\cite{johnstone1999note}. Under this coverage, a \textbf{cover} of an interval $[\ell,\ell']$ is generated by a partition into two closed intervals
\(([\ell,p],[p,\ell']).\)
This reflects the idea that information about a time interval can be reconstructed from information on adjacent subintervals.

\paragraph{Narratives}

As we alluded to earlier, narratives---our model for temporal data---consist of sheaves or cosheaves over time categories (viewed as sites, as above). The choice between sheaves and cosheaves is a modelling decision: should the model track data that \emph{persists} over time (sheaves) or data that \emph{accumulates} over time (cosheaves)?

\begin{definition}[$\cat{T}$-sheaves and $\cat{T}$-cosheaves]\label{prop:def:sheaves}
Let $\cat{T}$ be a time category equipped with the Johnstone coverage. If $\cat{D}$ has pullbacks, a \define{$\cat{D}$-valued sheaf on $\cat{T}$} is a presheaf $F:\cat{T}^{op}\to\cat{D}$ such that
\[
F([a,b])\cong F([a,p]) \times_{F([p,p])} F([p,b]).
\]
Dually, if $\cat{D}$ has pushouts, a \define{$\cat{D}$-valued cosheaf on $\cat{T}$} is a copresheaf $\hat F:\cat{T}\to\cat{D}$ such that for any interval $[a,b]$ and cover $([a,p],[p,b])$,
\[
\hat F([a,b])\cong \hat F([a,p]) +_{\hat F([p,p])} \hat F([p,b]).
\]
\end{definition}

\begin{definition}[Narratives]
Let $\cat{T}$ be a time category and $\cat{D}$ a category with pullbacks (resp. pushouts). The category of \define{persistent $\cat{D}$-narratives}, denoted $\mathsf{Pe}(\cat{T},\cat{D})$ (resp. \define{cumulative $\cat{D}$-narratives}, denoted $\mathsf{Cu}(\cat{T},\cat{D})$) consists of $\cat{D}$-valued sheaves (resp. cosheaves) on $\cat{T}$.
\end{definition}

A persistent narrative models information that must remain valid throughout a time interval. The sheaf condition ensures that such information over a larger interval can be reconstructed from data that \emph{persist} across overlapping sub-intervals. In contrast, a cumulative narrative models information that \emph{accumulates} over time, with the cosheaf condition expressing that data associated with a longer interval arises from aggregating the data of its sub-intervals. In both cases, a narrative encodes not only observations made at individual moments in time but also the relations describing how these observations evolve and interact across time intervals.

\begin{remark}
Since persistent narratives are sheaves, their categories can carry an internal intuitionistic logic. In particular, when the target category is \(\Set\), or more generally a presheaf category \(([\cat{C},\Set])\), categories of discrete persistent narratives form Grothendieck topoi. In~\cite{niu2026temporal}, this structure is used to define temporal truth values and a propositional language for reasoning about properties that hold at specific times or throughout intervals. The resulting logic is applied to public health models, where it can express temporal properties of individuals and conditions and support reasoning about possible routes of disease transmission in time-varying contact networks.
\end{remark}

\subsection{Changing Perspectives on Temporal Data:
The Persistence--\\Accumulation Adjunction}
\label{subsection:cumulative-and-persistent-perspectives}

Cumulative and persistent narratives are related by the following adjunction~\cite{theorytimevaryingdata2026}, which describes how one may change perspective between persistent and cumulative descriptions of temporal data. 
\begin{theorem}[Adjunction between cumulative and persistent narratives~\cite{theorytimevaryingdata2026}]
\label{thm:adjunction}
Let $\cat{D}$ be a category with limits and colimits and $\cat{T}$ a time category. There exist functors $\mathscr K:\mathsf{Pe}(\cat{T},\cat{D})\to\mathsf{Cu}(\cat{T},\cat{D})$ and $\mathscr P:\mathsf{Cu}(\cat{T},\cat{D})\to\mathsf{Pe}(\cat{T},\cat{D})$ forming an adjunction:
% https://q.uiver.app/#q=WzAsMixbMCwwLCJcXG1hdGhzZntQZX0oXFxjYXR7VH0sIFxcbWF0aHNme0R9KSJdLFsyLDAsIlxcbWF0aHNme0N1fShcXGNhdHtUfSwgXFxtYXRoc2Z7RH0pIl0sWzAsMSwie1xcbWF0aHNjcntLfX0iLDAseyJjdXJ2ZSI6LTJ9XSxbMSwwLCJ7XFxtYXRoc2Nye1B9fSIsMCx7ImN1cnZlIjotMn1dLFsyLDMsIiIsMix7ImxldmVsIjoxLCJzdHlsZSI6eyJuYW1lIjoiYWRqdW5jdGlvbiIsImJvZHkiOnsibmFtZSI6Im5vbmUifSwiaGVhZCI6eyJuYW1lIjoibm9uZSJ9fX1dXQ==
\[\begin{tikzcd}
	{\mathsf{Pe}(\cat{T}, \D)} && {\mathsf{Cu}(\cat{T}, \D).}
	\arrow[""{name=0, anchor=center, inner sep=0}, "{{\mathscr{K}}}", curve={height=-12pt}, from=1-1, to=1-3]
	\arrow[""{name=1, anchor=center, inner sep=0}, "{{\mathscr{P}}}", curve={height=-12pt}, from=1-3, to=1-1]
	\arrow["\dashv"{anchor=center, rotate=-90}, draw=none, from=0, to=1]
\end{tikzcd}\]
\end{theorem}
The functor \(\K\) sends a persistent narrative \(F:\T^{op}\to\D\) to its canonical cumulative counterpart \(\K F\). For each interval \([a,b]\in\T\), its value is defined by

\begin{equation}\label{eq:K-definition}
(\K F)_a^b
:=
\colim\Bigl(
(\T/[a,b])^{op}
\hookrightarrow
\T^{op}
\xrightarrow{F}
\D
\Bigr).
\end{equation}
Thus, \((\K F)_a^b\) accumulates the persistent data assigned by \(F\) to the subintervals of \([a,b]\). An inclusion \(i:[a,b]\hookrightarrow[c,d]\) induces a morphism
\[ \K F(i):(\K F)_a^b\longrightarrow(\K F)_c^d \]
by the universal property of the corresponding colimits. Thus, \(\K F(i)\) records how accumulated data over a smaller time window contribute to the accumulated data over a larger one. Dually, the functor \(\fP\) sends a cumulative narrative \(\hat F:\T\to\D\) to its canonical persistent counterpart. For each interval \([a,b]\in\T\), its value is defined by 

\begin{equation}\label{eq:P-definition}
(\fP\hat F)_a^b
:=
\lim\Bigl(
\T/[a,b]
\hookrightarrow
\T
\xrightarrow{\hat F}
\D
\Bigr).
\end{equation}
Thus, \((\fP\hat F)_a^b\) extracts the data compatible across the
subintervals of \([a,b]\). An inclusion \(i:[a,b]\hookrightarrow[c,d]\) induces a morphism 

\[ \fP\hat F(i): (\fP\hat F)_c^d \longrightarrow (\fP\hat F)_a^b \]
by the universal property of the corresponding limits. The adjunction is equipped with a unit and counit 
\[ \eta:\id_{\Pe}\Longrightarrow\fP\K \qquad\text{and}\qquad \varepsilon:\K\fP\Longrightarrow\id_{\Cu}. \]
For a persistent narrative \(F\), the component \( \eta_F:F\longrightarrow\fP\K F \) compares the prescribed persistent data with the persistent narrative reconstructed after passing through its cumulative completion. Dually, for a cumulative narrative \(\hat F\), the component \( \varepsilon_{\hat F}:\K\fP\hat F\longrightarrow\hat F \) compares the cumulative narrative reconstructed from its persistent completion with the original cumulative data.

The functors \(\K\) and \(\fP\) do not, in general, form an equivalence of categories. Consequently, passing from one perspective to the other and back need not recover the original temporal structure. Instead, the adjunction replaces the original data by the canonical objects determined by the corresponding colimit and limit constructions. The unit and counit therefore measure the extent to which the persistent and cumulative descriptions agree with these canonical reconstructions. A natural problem is therefore to characterize those temporal structures for which no information is lost. Equivalently, one may ask when the unit or counit of the adjunction is an isomorphism, or, more generally, which narratives are fixed points of the persistence--accumulation adjunction. The first research direction developed in Section~\ref{subsec:fixed-points} is devoted to these questions.

\subsection{Systematic Temporalization}\label{sec:temporalization}

We now explain how \cite{theorytimevaryingdata2026} addresses Desideratum~\ref{desideratum-3}, namely the requirement that a theory of temporal data provide systematic procedures for producing temporal analogues of familiar static notions. The guiding observation is that the static theories we care about (graphs, groups, spaces, etc.) are far more mature than their temporal counterparts, and that many static properties are most naturally expressed \emph{categorically}: either as membership in a subcategory, or as the existence of morphisms from a class of ``test objects'' (paths, cliques, colorings, and so on). Once phrased categorically, these notions admit canonical lifts to narrative categories.

\paragraph{Changing data types functorially.}
A first ingredient is that narratives are functorial in the data category. If $\cat{C}$ and $\cat{D}$ are categories of static data and $K:\cat{C}\to\cat{D}$ is a data-conversion functor, then postcomposition transports $\cat{C}$-valued narratives to $\cat{D}$-valued narratives whenever $K$ preserves the universal constructions that define the (co)sheaf conditions. Concretely, if $K$ is continuous, then for any time category $\cat{T}$ the assignment $F\mapsto K\circ F$ defines a functor $(K\circ-):\cat{Pe}(\cat{T},\cat{C})\to\cat{Pe}(\cat{T},\cat{D})$, and dually if $K$ preserves colimits then $(K\circ-)$ sends cumulative $\cat{C}$-narratives to cumulative $\cat{D}$-narratives. There is also a contravariant form: if $K:\cat{C}^{op}\to\cat{D}$ takes limits to colimits (resp.\ colimits to limits), then composition with $K$ switches perspectives, producing a functor from persistent to cumulative narratives (resp.\ cumulative to persistent). In this way, standard static constructions can be ``temporalized'' and transported across data types in a principled, functorial manner.

\paragraph{Lifting static properties by change of base.}
The second ingredient is a systematic way to lift \emph{classes of static objects} to \emph{classes of narratives}. Suppose a static property is presented by a subcategory inclusion $P:\cat{P}\hookrightarrow\cat{C}$ (one should think of $\cat{P}$ as the full subcategory of objects of $\cat{C}$ satisfying the property). If $P$ is continuous, then the induced functor $(P\circ-)$ identifies those persistent $\cat{C}$-narratives whose values lie in $\cat{P}$ in a way compatible with the sheaf condition; dually, if $P$ preserves colimits, it specifies the corresponding class of cumulative narratives. This reproduces, in the temporal setting, the common pattern from static combinatorics and algebra: properties become subcategories, and temporal analogues become subcategories of narrative categories.

\paragraph{Temporal analogues at a chosen resolution.}
While change of base yields a clean lift, it can be too strict in applications: one may only care that a property appears at a particular temporal granularity. The framework in~\cite{theorytimevaryingdata2026}  accounts for this by introducing a functorial method for changing temporal resolution. Given a sub-join-semilattice inclusion $\tau:\cat{S}\hookrightarrow\cat{T}$, precomposition defines a restriction functor $(-\circ\tau)$ sending $\cat{T}$-narratives to $\cat{S}$-narratives (persistent-to-persistent and cumulative-to-cumulative). By combining restriction with change of base, one obtains a general notion of a narrative \emph{satisfying a static property only on the intervals in $\cat{S}$}: rather than requiring the property at all intervals, one imposes it after restricting to $\cat{S}$, and defines the resulting class of narratives by a pullback that expresses compatibility between ``restriction to $\cat{S}$'' and ``landing in $\cat{P}$''. This produces a family of temporal analogues parametrized by resolution, reflecting the empirical reality that certain phenomena are only visible after aggregating over sufficiently large time windows.

\begin{example}[Temporal resolution]
Recall from Example~\ref{ex:changing-temporal-resolution} that a partition
\(a=r_0\le\cdots\le r_n=b\) determines a functor \(R:\cat T_n\longrightarrow\cat T/[a,b]\). Consequently, every \(\D\)-valued narrative
\(F:(\cat T/[a,b])^\op\to\D\)
restricts along \(R\) to the narrative
\[
F\circ R^\op:\cat T_n^\op\longrightarrow\D.
\]
The resulting narrative is the restriction of \(F\) to the temporal resolution induced by the chosen partition.
\end{example}

\paragraph{Graph-theoretic case studies: paths, cliques, and dualities.}
To demonstrate the machinery recovers familiar temporal notions while adding conceptual clarity, \cite{theorytimevaryingdata2026} develops graph-theoretic case studies. Path-like behavior is obtained by lifting the subcategory of paths in $\cat{Grph}$ to a corresponding class of graph narratives; temporal paths in a graph narrative can then be expressed as subobjects of that narrative, recovering the static viewpoint in which many decision problems are homomorphism problems. Temporal cliques are treated similarly: by lifting the subcategory of complete graphs and combining it with a resolution restriction (e.g., requiring completeness only over intervals of length at least $n$), one recovers standard definitions of temporal $k$-cliques from the temporal-graph literature, but now characterized by morphisms of narratives. This categorical reformulation also exposes dualities: just as cliques and colorings are related by categorical duality in the static setting, temporal analogues inherit corresponding dual notions, and these dualities depend on whether one works cumulatively or persistently.

A key conceptual payoff is that these temporalizations are not ad hoc: they arise from a small number of functorial principles (change of base, change of resolution, and categorical characterizations by morphisms). Moreover, the interaction with the persistent--cumulative adjunction highlights genuine temporal subtleties: for instance, properties defined by stability under pushouts may behave differently from those defined by stability under pullbacks, and changing perspective can therefore change which temporal analogues exist or are well behaved.

\section{Three Vignettes on Time-Varying Data and the Ensuing Research Directions}\label{sec:directions}

\subsection{Vignette 1: Narratives and the Fixed Points of the Persistence--Accumulation Adjunction}\label{subsec:fixed-points}
The functor \(\K\) sends a persistent narrative 
\(F\colon \cat{T}^{op}\to\cat{D}\) 
to its cumulative counterpart \(\K F\) by forming pushouts that encode data accumulated over a given time window. 
Applying \(\fP \) then returns to the persistent viewpoint by computing pullbacks, producing \(\fP \K F\). 
In general, this pushout--pullback round trip does not recover the original persistent sheaf. 
More conceptually, because \(\K \) and \(\fP \) form an adjunction rather than an equivalence, the composite \(\fP \K \) need not act as the identity on persistent narratives. 
Equivalently, the unit 
\(\eta_F\colon F\to\fP \K F\) 
may fail to be an isomorphism. Thus, passing between these viewpoints may approximate the original temporal structure by its associated limit and colimit constructions. Figure~\ref{fig:counit_example} illustrates this phenomenon by exhibiting a persistent narrative for which the unit \(\eta_F:F\to\fP \K F\) is not an isomorphism.

\begin{figure}[hbt]
    \centering
    \includegraphics[width=\linewidth]{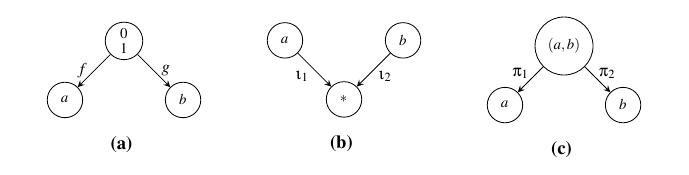}
    \caption{An example showing that the adjunction \(\K \dashv\fP \) is not an equivalence. (a) Let \(\cat{T}\) be the time category \([0,0]\hookrightarrow[0,1]\hookleftarrow[1,1]\). Consider the persistent narrative \(F\colon\cat{T}^{op}\to\cat{Set}\) shown in the figure. Its value at the interval \([0,1]\) is \(F_0^1=\{0,1\}\). (b) Applying \(\K \) yields a cumulative narrative with \((\K F)_0^1=F_0^0+_{F_0^1}F_1^1=\{a\}+_{\{0,1\}}\{b\}\), which in this case is the singleton \(\{\ast\}\). (c) Applying \(\fP \) produces a persistent narrative with \((\fP \K F)_0^1=F_0^0\times_{(\K F)_0^1}F_1^1\cong\{(a,b)\}\). Since \(\{(a,b)\}\not\cong\{0,1\}\), the component \(\eta_{F_0^1}:F_0^1\to(\fP \K F)_0^1\) is not an isomorphism. Hence \(\fP \K F\not\cong F\).}
    \label{fig:counit_example}
\end{figure}

The significance of Figure~\ref{fig:counit_example} is not merely that the unit fails to be invertible, but why it fails. The span \(F_0^0\xleftarrow{}F_0^1\xrightarrow{}F_1^1\) defining the persistent narrative is \emph{prescribed data} and therefore need not satisfy any universal property. By contrast, the span \(F_0^0\xleftarrow{}F_0^0\times_{(\K F)_0^1}F_1^1\xrightarrow{}F_1^1\) is \emph{built} by taking the pullback of the cospan \(F_0^0\to(\K F)_0^1\leftarrow F_1^1\). The unit
\[
\eta_{F_0^1}:F_0^1\longrightarrow (\fP \K F)_0^1
\]
is the canonical comparison between the prescribed persistent data and its canonical approximation via the pushout--pullback round trip. 

This observation leads to a natural characterization of the fixed points of the adjunction in the direction \(F\to\fP \K F\). By a \emph{fixed point} we mean a persistent narrative \(F\) for which the unit \(\eta_F:F\to\fP \K F\) is an isomorphism. Equivalently, as illustrated by the diagram below, this requires the span \(F_0^0\xleftarrow{}F_0^1\xrightarrow{}F_1^1\) to coincide with the pullback of its pushout \(F_0^0\to(\K F)_0^1\leftarrow F_1^1\). 

% https://q.uiver.app/#q=WzAsNCxbMSwwLCJGXzBeMSJdLFsyLDEsIkZfMV4xIl0sWzAsMSwiRl8wXjAiXSxbMSwyLCIoXFxLIEYpXzBeMSJdLFswLDFdLFswLDJdLFsxLDNdLFsyLDNdLFszLDAsIiIsMSx7InN0eWxlIjp7Im5hbWUiOiJjb3JuZXIifX1dLFswLDMsIiIsMSx7ImxhYmVsX3Bvc2l0aW9uIjoyMCwic3R5bGUiOnsibmFtZSI6ImNvcm5lciJ9fV1d
\[\begin{tikzcd}
	& {F_0^1} & \\
	{F_0^0} && {F_1^1} \\
	& {(\K F)_0^1}
	\arrow[from=1-2, to=2-1]
	\arrow[from=1-2, to=2-3]
	\arrow["\lrcorner"{anchor=center, pos=0, rotate=-45}, draw=none, from=1-2, to=3-2]
	\arrow[from=2-1, to=3-2]
	\arrow[from=2-3, to=3-2]
	\arrow["\lrcorner"{anchor=center, pos=0, rotate=135}, draw=none, from=3-2, to=1-2]
\end{tikzcd}\]

    This characterization immediately yields a broad class of fixed points whenever the ambient category \(\D\) is adhesive. Recall that in an adhesive category, pushouts along monomorphisms are Van Kampen~\cite{LackAdhesive}, and hence are stable under pullback. Consequently, if the span \(F_0^0\xleftarrow{}F_0^1\xrightarrow{}F_1^1\) consists of monomorphisms, then the associated pushout square is also a pullback. Therefore, \(F\) is a fixed point of the adjunction in the direction \(F\to\fP \K F\).

The discussion above concerns the unit of the adjunction and the recovery of persistent data. Dually, one may study the counit and the recovery of cumulative data. Figure~\ref{fig:counit_example} already shows that these two behaviors need not agree: panel~(b) is recovered after the pullback--pushout round trip, whereas panel~(a) is not recovered after the pushout--pullback round trip. This asymmetry motivates the introduction of narratives, which encode persistent and cumulative descriptions simultaneously together with the comparison between them.

\subsubsection{Narratives encoding persistence and accumulation simultaneously}

\begin{definition}[Cotwisted Arrow Category]
Let $\cat{D}$ be a small category. The \emph{cotwisted arrow category} of $\cat{D}$, denoted $\cotwD$, is defined as follows. 

\begin{itemize}
  \item \textbf{Objects} are morphisms $f \colon x \to y$ in $\cat{D}$.
  \item \textbf{Morphisms} from $(x \xrightarrow{f} y)$ to $(x' \xrightarrow{f'} y')$ are pairs of morphisms $(u \colon x \to x', v \colon y' \to y)$ in $\cat{D}$ such that the following diagram commutes:
% https://q.uiver.app/#q=WzAsNCxbMCwwLCJ4Il0sWzAsMiwieSJdLFsyLDAsIngnIl0sWzIsMiwieSciXSxbMCwxLCJmIiwyXSxbMiwzXSxbMCwyLCJ1Il0sWzMsMSwidiJdXQ==
\[\begin{tikzcd}
	x && {x'} \\
	\\
	y && {y'}
	\arrow["u", from=1-1, to=1-3]
	\arrow["f"', from=1-1, to=3-1]
	\arrow["f'", from=1-3, to=3-3]
	\arrow["v", from=3-3, to=3-1]
\end{tikzcd}
\quad \text{that is, } f = v \circ f' \circ u.
\]
Thus, a morphism in $\cotwD$ encodes a factorization of \(f\) through \(f'\).  The notation \(\cotwD\) is mnemonic for the defining commutative square: the superscript \(\triangleright\) reminds the reader that the domain morphism points to the right, whereas the subscript \(\triangleleft\) indicates that the codomain morphism points to the left.
  \item \textbf{Composition} is defined componentwise:
  \[
  (u', v') \circ (u, v) = (u' \circ u, v \circ v').
  \]
\end{itemize}
\end{definition}

\begin{lemma}\label{lem:cotwisted-pullbacks}
If \(\D\) has pullbacks and pushouts, then the cotwisted arrow category
\(\cotwD\) has pullbacks.
\end{lemma}
\begin{proof}
Consider two morphisms in \(\cotwD\) with common codomain,
\[
(u_i,v_i):
(x_i\xrightarrow{f_i}y_i)
\longrightarrow
(x_0\xrightarrow{f_0}y_0),
\qquad i=1,2,
\]
satisfying \(f_i=v_i\circ f_0\circ u_i\). They are depicted in the following diagram.

% https://q.uiver.app/#q=WzAsNixbMSwxLCJ4XzAiXSxbMCwwLCJ4XzEiXSxbMiwwLCJ4XzIiXSxbMCwzLCJ5XzEiXSxbMiwzLCJ5XzIiXSxbMSwyLCJ5XzAiXSxbMSwwLCJ1XzEiLDJdLFsyLDAsInVfMiJdLFs1LDMsInZfMSIsMl0sWzUsNCwidl8yIl0sWzEsMywiZl8xIiwyXSxbMiw0LCJmXzIiLDJdLFswLDUsImZfMCIsMl0sWzEwLDEyLCI9IiwxLHsic2hvcnRlbiI6eyJzb3VyY2UiOjIwLCJ0YXJnZXQiOjIwfSwic3R5bGUiOnsiYm9keSI6eyJuYW1lIjoibm9uZSJ9LCJoZWFkIjp7Im5hbWUiOiJub25lIn19fV0sWzExLDEyLCI9IiwxLHsic2hvcnRlbiI6eyJzb3VyY2UiOjIwLCJ0YXJnZXQiOjIwfSwic3R5bGUiOnsiYm9keSI6eyJuYW1lIjoibm9uZSJ9LCJoZWFkIjp7Im5hbWUiOiJub25lIn19fV1d
\[\begin{tikzcd}
	{x_1} && {x_2} \\
	& {x_0} \\
	& {y_0} \\
	{y_1} && {y_2}
	\arrow["{u_1}"', from=1-1, to=2-2]
	\arrow[""{name=0, anchor=center, inner sep=0}, "{f_1}"', from=1-1, to=4-1]
	\arrow["{u_2}", from=1-3, to=2-2]
	\arrow[""{name=1, anchor=center, inner sep=0}, "{f_2}"', from=1-3, to=4-3]
	\arrow[""{name=2, anchor=center, inner sep=0}, "{f_0}"', from=2-2, to=3-2]
	\arrow["{v_1}"', from=3-2, to=4-1]
	\arrow["{v_2}", from=3-2, to=4-3]
	\arrow["{=}"{description}, draw=none, from=0, to=2]
	\arrow["{=}"{description}, draw=none, from=1, to=2]
\end{tikzcd}\]

Form in \(\D\) the pullback of the span \(x_1\xrightarrow{\,u_1\,}x_0\xleftarrow{\,u_2\,}x_2\), namely \(x_1\xleftarrow{\,p_1\,}P\xrightarrow{\,p_2\,}x_2\), and the pushout of the cospan \(y_1\xleftarrow{\,v_1\,}y_0\xrightarrow{\,v_2\,}y_2\), namely \(y_1\xrightarrow{\,q_1\,}Q\xleftarrow{\,q_2\,}y_2\). Since \(u_1\circ p_1=u_2\circ p_2\) and \(q_1\circ v_1=q_2\circ v_2\), we obtain
\[
q_1\circ f_1\circ p_1
=
q_1\circ v_1\circ f_0\circ u_1\circ p_1
=
q_2\circ v_2\circ f_0\circ u_2\circ p_2
=
q_2\circ f_2\circ p_2.
\]
Hence, by the universal property of the pushout, there exists a unique morphism \(h:P\to Q\) satisfying \(h=q_1\circ f_1\circ p_1=q_2\circ f_2\circ p_2\). Consequently, \((p_1,q_1):(P\xrightarrow{h}Q)\to(x_1\xrightarrow{f_1}y_1)\) and \((p_2,q_2):(P\xrightarrow{h}Q)\to(x_2\xrightarrow{f_2}y_2)\) are morphisms in \(\cotwD\) whose composites with \((u_1,v_1)\) and \((u_2,v_2)\) coincide.

To verify the universal property, let \((a_i,b_i):(z\xrightarrow{g}w)\to(x_i\xrightarrow{f_i}y_i)\), \(i=1,2\), be morphisms in \(\cotwD\) whose composites with \((u_1,v_1)\) and \((u_2,v_2)\) agree. Then \(u_1\circ a_1=u_2\circ a_2\) and \(b_1\circ v_1=b_2\circ v_2\). By the universal property of the pullback, there exists a unique morphism \(a:z\to P\) such that \(p_i\circ a=a_i\), \(i=1,2\). Similarly, by the universal property of the pushout, there exists a unique morphism \(b:Q\to w\) such that \(b\circ q_i=b_i\), \(i=1,2\). Moreover,
\[
b\circ h\circ a
=
b\circ q_i\circ f_i\circ p_i\circ a
=
b_i\circ f_i\circ a_i
=
g,
\]
so \((a,b):(z\xrightarrow{g}w)\to(P\xrightarrow{h}Q)\) is a morphism in \(\cotwD\). Its uniqueness follows from the uniqueness of \(a\) and \(b\). Therefore, \((P\xrightarrow{h}Q)\), together with the morphisms \((p_1,q_1)\) and \((p_2,q_2)\), is the pullback of \((u_1,v_1)\) and \((u_2,v_2)\) in \(\cotwD\).
\end{proof}

We can now define \(\cat{T}\)-sheaves on \(\cotwD\) as follows.

\begin{proposition}[\(\cat{T}\)-sheaves on \(\cotwD\)]
Let \(\T\) be any time category equipped with the Johnstone coverage. Suppose that \(\D\) has limits and colimits (and therefore \(\cotwD\) has limits).  Then a \(\cotwD\)-valued sheaf is a presheaf \(X : \T^{op} \to \cotwD\) such that: for any interval \([a,b]\) and any cover \(([a, p],[p, b])\) of this interval, \(X([a,b])\) is the pullback \(X([a,p]) \times_{X([p,p])} X([p,b])\).
\end{proposition}
\begin{proof}
Since the Johnstone coverage is generated by binary covers, the sheaf condition requires precisely the existence of the corresponding pullbacks in \(\cotwD\). By Lemma~\ref{lem:cotwisted-pullbacks}, \(\cotwD\) has these pullbacks whenever \(\D\) has pullbacks and pushouts. The result therefore follows directly from the definition of a sheaf.
\end{proof}

\begin{definition}\label{def:narratives-category}
  We denote by \(\cat{Nar}(\T, \cotwD)\) the category of \(\cotwD\)-valued sheaves on \(\T\) and we call it the category of \define{\(\cotwD\)-narratives} with \(\cat{T}\)-time.  
\end{definition}

A narrative \(X\in\Nar\) assigns to every interval \([a,b]\) a morphism \(X_a^b=(P_a^b\xrightarrow{\gamma_a^b}C_a^b)\) in \(\D\). The objects \(P_a^b\) form its \emph{persistent component}, the objects \(C_a^b\) form its \emph{cumulative component}, and the morphisms \(\gamma_a^b\) compare these two descriptions. Figure~\ref{fig:narrative-shape} illustrates this structure on the category \(\cat{I}_{\mathbb{N}}/[0,2]\), whose objects are the intervals contained in \([0,2]\subseteq\mathbb{N}\).

\begin{remark}\label{rem:compar-morph-need-not-be-iso}
The comparison morphisms \(\gamma_a^b\) need not be isomorphisms, even when \(a=b\). Thus the persistent and cumulative descriptions of the same interval generally contain different kinds of information. For example, in \(\Set\) or \(\Grph\), one may interpret \(\gamma_t^t:P_t^t\to C_t^t\) as an attribute map assigning to each element or vertex of \(P_t^t\) a color, label, weight, community, role, or other feature in \(C_t^t\). In this case, \(P_t^t\) represents the collection of objects and \(C_t^t\) represents the collection of possible attributes, with \(\gamma_t^t\) recording the assignment of attributes to objects. The sheaf structure then encodes how both the objects and their attributes relate across overlapping intervals of time: restrictions on the persistent side track how objects change through time, while restrictions on the cumulative side track the compatibility and evolution of their attributes. The sheaf condition ensures that local descriptions on overlapping intervals can be uniquely assembled into a coherent global account of both the objects and their attributes.
\end{remark}

\begin{figure}[hbt] \centering
% https://q.uiver.app/#q=WzAsMTIsWzIsMCwiUF8wXjIiXSxbMSwxLCJQXzBeMSJdLFszLDEsIlBfMV4yIl0sWzAsMiwiUF8wXjAiXSxbMiwyLCJQXzFeMSJdLFs0LDIsIlBfMl4yIl0sWzAsMywiQ18wXjAiXSxbMiwzLCJDXzFeMSJdLFs0LDMsIkNfMl4yIl0sWzEsNCwiQ18wXjEiXSxbMyw0LCJDXzFeMiJdLFsyLDUsIkNfMF4yIl0sWzAsMV0sWzAsMl0sWzAsNCwiIiwwLHsic3R5bGUiOnsibmFtZSI6ImNvcm5lciIsImJvZHkiOnsibmFtZSI6Im5vbmUifSwiaGVhZCI6eyJuYW1lIjoibm9uZSJ9fX1dLFsxLDNdLFsxLDRdLFsyLDRdLFsyLDVdLFs2LDldLFs3LDldLFs3LDEwXSxbOCwxMF0sWzksMTFdLFsxMCwxMV0sWzExLDcsIiIsMCx7InN0eWxlIjp7Im5hbWUiOiJjb3JuZXIiLCJib2R5Ijp7Im5hbWUiOiJub25lIn0sImhlYWQiOnsibmFtZSI6Im5vbmUifX19XSxbMyw2LCJ7XFxnYW1tYV8wXjB9IiwyLHsiY3VydmUiOi0xLCJzdHlsZSI6eyJib2R5Ijp7Im5hbWUiOiJkYXNoZWQifX19XSxbNCw3LCJ7XFxnYW1tYV8xXjF9IiwyLHsiY3VydmUiOi0xLCJzdHlsZSI6eyJib2R5Ijp7Im5hbWUiOiJkYXNoZWQifX19XSxbNSw4LCJ7XFxnYW1tYV8yXjJ9IiwwLHsiY3VydmUiOi0xLCJzdHlsZSI6eyJib2R5Ijp7Im5hbWUiOiJkYXNoZWQifX19XSxbMSw5LCJ7XFxnYW1tYV8wXjF9IiwyLHsiY3VydmUiOi0yLCJzdHlsZSI6eyJib2R5Ijp7Im5hbWUiOiJkYXNoZWQifX19XSxbMiwxMCwie1xcZ2FtbWFfMV4yfSIsMCx7ImN1cnZlIjotMiwic3R5bGUiOnsiYm9keSI6eyJuYW1lIjoiZGFzaGVkIn19fV0sWzAsMTEsIntcXGdhbW1hXzBeMn0iLDAseyJjdXJ2ZSI6LTQsInN0eWxlIjp7ImJvZHkiOnsibmFtZSI6ImRhc2hlZCJ9fX1dLFs5LDEsIiIsMSx7InN0eWxlIjp7Im5hbWUiOiJjb3JuZXIifX1dLFsxMCwyLCIiLDEseyJzdHlsZSI6eyJuYW1lIjoiY29ybmVyIn19XV0=
\[\begin{tikzcd}
	&& {P_0^2} && \\
	& {P_0^1} && {P_1^2} \\
	{P_0^0} && {P_1^1} && {P_2^2} \\
	{C_0^0} && {C_1^1} && {C_2^2} \\
	& {C_0^1} && {C_1^2} \\
	&& {C_0^2}
	\arrow[from=1-3, to=2-2]
	\arrow[from=1-3, to=2-4]
	\arrow["\lrcorner"{anchor=center, pos=0.125, rotate=-45}, draw=none, from=1-3, to=3-3]
	\arrow["{{\gamma_0^2}}", curve={height=-24pt}, dashed, from=1-3, to=6-3]
	\arrow[from=2-2, to=3-1]
	\arrow[from=2-2, to=3-3]
	\arrow["{{\gamma_0^1}}"', curve={height=-12pt}, dashed, from=2-2, to=5-2]
	\arrow[from=2-4, to=3-3]
	\arrow[from=2-4, to=3-5]
	\arrow["{{\gamma_1^2}}", curve={height=-12pt}, dashed, from=2-4, to=5-4]
	\arrow["{{\gamma_0^0}}"', curve={height=-6pt}, dashed, from=3-1, to=4-1]
	\arrow["{{\gamma_1^1}}"', curve={height=-6pt}, dashed, from=3-3, to=4-3]
	\arrow["{{\gamma_2^2}}", curve={height=-6pt}, dashed, from=3-5, to=4-5]
	\arrow[from=4-1, to=5-2]
	\arrow[from=4-3, to=5-2]
	\arrow[from=4-3, to=5-4]
	\arrow[from=4-5, to=5-4]
	\arrow["\lrcorner"{anchor=center, pos=0.125, rotate=135}, draw=none, from=5-2, to=2-2]
	\arrow[from=5-2, to=6-3]
	\arrow["\lrcorner"{anchor=center, pos=0.125, rotate=135}, draw=none, from=5-4, to=2-4]
	\arrow[from=5-4, to=6-3]
	\arrow["\lrcorner"{anchor=center, pos=0.125, rotate=135}, draw=none, from=6-3, to=4-3]
\end{tikzcd}\]
\caption{A \(\D\)-narrative on the category \(\cat{I}_{\mathbb{N}}/[0,2]\), whose objects are the intervals contained in \([0,2]\subseteq\mathbb{N}\).  The upper diagram is its persistent component, the lower diagram is its cumulative component, and the dashed morphisms \(\gamma_a^b:P_a^b\to C_a^b\) compare the two. In the case shown, the persistent data \(P_0^0 \leftarrow P_0^1 \rightarrow P_1^1 \leftarrow P_1^2 \rightarrow P_2^2,\) is prescribed data, while \(P_0^2\) is determined by a pullback. On the cumulative side, \(C_0^1,C_1^2\), and \(C_0^2\) are determined by pushouts.} \label{fig:narrative-shape}
\end{figure}
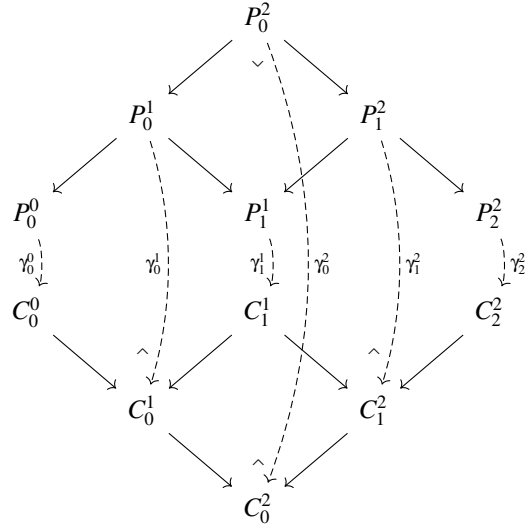

We have constructed the category \(\Nar\), whose objects simultaneously encode both the persistent and cumulative descriptions of a temporal object together with the comparison between them. As illustrated in Figure~\ref{fig:narrative-shape}, the persistent and cumulative components may each arise either from prescribed data (see Remark~\ref{rem:compar-morph-need-not-be-iso}) or from their respective universal constructions. These independent possibilities will later give rise to the classification of narratives by rigidity (Section~\ref{subsubsec:rigidity}).

\subsubsection{Factorizing the persistence--accumulation adjunction}

Having constructed the category of narratives, we now show that it naturally sits between persistent and cumulative narratives by factorizing the persistence–accumulation adjunction.

\begin{theorem}\label{thm:adjunction-factorization}
    The adjunction \(\K \dashv \fP\) factorizes through the category \(\cat{Nar}(\T, \cotwD)\) via the functors \(\cotwP\) and \(\cotwK \). That is, both the outer and inner triangles in the following diagram commute:

% https://q.uiver.app/#q=WzAsMyxbNCwwLCJcXEN1Il0sWzIsNCwiXFxjYXR7TmFyfShcXFQsIFxcY290d0QpIl0sWzAsMCwiXFxQZSJdLFswLDEsIlxcY290d3tcXGZQfSIsMix7Im9mZnNldCI6LTIsImN1cnZlIjoyLCJzaG9ydGVuIjp7InRhcmdldCI6MTB9fV0sWzIsMCwiXFxLIiwwLHsiY3VydmUiOi0yfV0sWzEsMiwiXFxjb3R3ZG9tIiwyLHsib2Zmc2V0IjotMiwiY3VydmUiOjIsInNob3J0ZW4iOnsic291cmNlIjoxMH19XSxbMiwxLCJcXGNvdHd7XFxLfSIsMix7Im9mZnNldCI6MywiY3VydmUiOjIsInNob3J0ZW4iOnsidGFyZ2V0IjoxMH19XSxbMSwwLCJcXGNvdHdjb2QiLDIseyJvZmZzZXQiOjMsImN1cnZlIjoyLCJzaG9ydGVuIjp7InNvdXJjZSI6MTB9fV0sWzAsMiwiXFxmUCIsMCx7Im9mZnNldCI6LTEsImN1cnZlIjotMn1dLFs0LDgsIiIsMCx7ImxldmVsIjoxLCJzdHlsZSI6eyJuYW1lIjoiYWRqdW5jdGlvbiJ9fV1d
\[\begin{tikzcd}
	\Pe &&&& \Cu \\
	\\
	\\
	\\
	&& {\cat{Nar}(\T, \cotwD)}
	\arrow[""{name=0, anchor=center, inner sep=0}, "\K", curve={height=-12pt}, from=1-1, to=1-5]
	\arrow["{\cotw{\K}}"', shift right=3, curve={height=12pt}, between={0}{0.9}, from=1-1, to=5-3]
	\arrow[""{name=1, anchor=center, inner sep=0}, "\fP", shift left, curve={height=-12pt}, from=1-5, to=1-1]
	\arrow["{\cotw{\fP}}"', shift left=2, curve={height=12pt}, between={0}{0.9}, from=1-5, to=5-3]
	\arrow["\cotwdom"', shift left=2, curve={height=12pt}, between={0.1}{1}, from=5-3, to=1-1]
	\arrow["\cotwcod"', shift right=3, curve={height=12pt}, between={0.1}{1}, from=5-3, to=1-5]
	\arrow["\dashv"{anchor=center, rotate=-90}, draw=none, from=0, to=1]
\end{tikzcd}\]

Equivalently, 
\[
\cotwcod\circ\cotwK=\K,
\qquad
\cotwdom\circ\cotwP=\fP.
\]
Moreover, 
\[
\cotwdom\circ\cotwK=\id_{\Pe},
\qquad
\cotwcod\circ\cotwP=\id_{\Cu}.
\]
\end{theorem}

The forgetful functors recover the persistent and cumulative components of a narrative by forgetting one side of the comparison morphism:
\[
\cotwdom :\Nar\longrightarrow\Pe,
\qquad
\cotwcod:\Nar\longrightarrow\Cu.
\]
Conversely, the functors
\[
\cotw{\K}:\Pe\longrightarrow\Nar,
\qquad
\cotw{\fP}:\Cu\longrightarrow\Nar,
\]
canonically complete persistent and cumulative narratives into \(\cotwD\)-valued narratives by adjoining their cumulative and persistent counterparts, respectively.

We now construct each of these four functors. Once these constructions are in place, the proof of Theorem~\ref{thm:adjunction-factorization} will follow by verifying that the four triangles commute by construction.

\paragraph{The forgetful functors.}\label{paragraph:forgetful-fun}

We first construct the functors
\[
\cotwdom :\Nar\longrightarrow\Pe,
\qquad
\cotwcod:\Nar\longrightarrow\Cu.
\]
These arise from the canonical domain and codomain projections of the cotwisted arrow category,
\[
\D
\xleftarrow{\ \dom }
\cotwD
\xrightarrow{\ \cod }
\D^{op},
\]
where \(\dom\) sends an arrow \(x\xrightarrow{f}y\) to its domain \(x\), while \(\cod\) sends it to its codomain \(y\). On a morphism \((u,v):(x\xrightarrow{f}y)\to(x'\xrightarrow{f'}y')\) in \(\cotwD\), they are given by

% https://q.uiver.app/#q=WzAsOCxbMywwLCJ4Il0sWzMsMiwieSJdLFs1LDAsIngnIl0sWzUsMiwieSciXSxbNywxLCJ5Il0sWzgsMSwieSciXSxbMCwxLCJ4Il0sWzEsMSwieCciXSxbMCwxLCJmIiwyXSxbMiwzLCJmJyJdLFswLDIsInUiXSxbMywxLCJ2IiwyXSxbNCw1LCJ2Il0sWzYsNywidSJdLFs5LDQsIlxcY29kIiwwLHsic2hvcnRlbiI6eyJzb3VyY2UiOjMwLCJ0YXJnZXQiOjIwfSwibGV2ZWwiOjEsInN0eWxlIjp7InRhaWwiOnsibmFtZSI6Im1hcHMgdG8ifX19XSxbOCw3LCJcXGRvbSIsMix7InNob3J0ZW4iOnsic291cmNlIjozMCwidGFyZ2V0IjoyMH0sImxldmVsIjoxLCJzdHlsZSI6eyJ0YWlsIjp7Im5hbWUiOiJtYXBzIHRvIn19fV1d
\[\begin{tikzcd}
	&&& x && {x'} &&& \\
	x & {x'} &&&&&& y & {y'} \\
	&&& y && {y'}
	\arrow["u", from=1-4, to=1-6]
	\arrow[""{name=0, anchor=center, inner sep=0}, "f"', from=1-4, to=3-4]
	\arrow[""{name=1, anchor=center, inner sep=0}, "{f'}", from=1-6, to=3-6]
	\arrow["u", from=2-1, to=2-2]
	\arrow["v", from=2-8, to=2-9]
	\arrow["v"', from=3-6, to=3-4]
	\arrow["\dom"', between={0.3}{0.8}, maps to, from=0, to=2-2]
	\arrow["\cod", between={0.3}{0.8}, maps to, from=1, to=2-8]
\end{tikzcd}\]

Thus \(\dom\) is covariant, whereas \(\cod\) is naturally viewed as taking values in \(\D^{op}\).

\begin{proposition}\label{prop:dom-cod-functors}
Composition with \(\dom\) and \(\cod\) defines functors
\[
\cotwdom :\Nar\longrightarrow\Pe,
\qquad
\cotwcod:\Nar\longrightarrow\Cu
\]
given on objects by 
\[\cotwdom(X)=\dom\circ X,
\qquad
\cotwcod(X)=\cod\circ X.\]
\end{proposition}

\begin{proof}
Let \(X:\T^{op}\to\cotwD\) be a narrative. For each interval \([a,b]\), write
\[
X_a^b=\left(P_a^b\xrightarrow{\gamma_a^b}C_a^b\right).
\]
We define \(\cotwdom(X)\) on objects by \([a,b]\mapsto P_a^b\) and \(\cotwcod(X)\) by \([a,b]\mapsto C_a^b\). Given a morphism \(i:[a,b]\hookrightarrow[c,d]\), if \(X(i)=(p_i,c_i)\), we put \(\cotwdom(X)(i)=p_i\) and \(\cotwcod(X)(i)=c_i\). Since composition in \(\cotwD\) is defined componentwise, both assignments preserve identities and composition. Furthermore, because \(X\) is a sheaf valued in \(\cotwD\), applying \(\dom\) to each pullback diagram recovers the sheaf condition defining persistent narratives, while applying \(\cod\) recovers the corresponding pushout condition defining cumulative narratives. Thus, \(\cotwdom(X)\in\Pe\) and \(\cotwcod(X)\in\Cu\). Finally, a morphism of narratives is a natural transformation in \(\cotwD\), and composing it componentwise with \(\dom\) or \(\cod\) yields natural transformations between the corresponding persistent or cumulative narratives. Therefore, \(\cotwdom\) and \(\cotwcod\) define functors.
\end{proof}

\paragraph{The cumulative completion.}

Having constructed the forgetful functors, we now turn to the opposite direction. Our next goal is to show that every persistent narrative admits a canonical cumulative completion into a \(\cotwD\)-valued narrative. 

\begin{proposition}[Cumulative completion of a persistent narrative]\label{prop:cum-completion}
There is a functor
\[
\cotwK \colon \Pe \to \Nar
\]
that sends a persistent narrative \(F:\T^{op}\to\D\) to the narrative
\[
\cotwK F\colon \T^{op}\to \cotwD.
\]
defined as follows. Given an interval \([a,b]\in\T\), \(\cotwK F\) maps it to the object
\[
F_a^b \xrightarrow{\alpha_a^b} \K F_a^b
\]
of \(\cotwD\), where \(\alpha_a^b\) is the canonical cocone map into the colimit defining \(\K F_a^b\). Given an inclusion \(i\colon [a,b]\hookrightarrow [c,d]\), \(\cotwK F\) maps it to the morphism in \(\cotwD\)
\[
\bigl(
F_c^d \xrightarrow{F(i)} F_a^b,\;
\K F_a^b \xrightarrow{\K F(i)} \K F_c^d
\bigr),
\]
which satisfies the cotwisted commutativity condition:
% https://q.uiver.app/#q=WzAsNCxbMSwwLCJYX2NeZCAiXSxbMSwzLCJcXEsgWF9jXmQiXSxbMCwxLCJYX2FeYiJdLFswLDIsIlxcSyBYX2FeYiJdLFswLDEsIlxcYWxwaGFfY15kICJdLFsyLDMsIlxcYWxwaGFfYV5iIiwyXSxbMCwyLCJYKGkpIiwyXSxbMywxLCJcXEsgWChpKSIsMl0sWzUsNCwiPSIsMSx7InNob3J0ZW4iOnsic291cmNlIjoyMCwidGFyZ2V0IjoyMH0sInN0eWxlIjp7ImJvZHkiOnsibmFtZSI6Im5vbmUifSwiaGVhZCI6eyJuYW1lIjoibm9uZSJ9fX1dXQ==
\[\begin{tikzcd}
	& {F_c^d } \\
	{F_a^b} \\
	{\K F_a^b} \\
	& {\K F_c^d}
	\arrow["{F(i)}"', from=1-2, to=2-1]
	\arrow[""{name=0, anchor=center, inner sep=0}, "{\alpha_c^d }", from=1-2, to=4-2]
	\arrow[""{name=1, anchor=center, inner sep=0}, "{\alpha_a^b}"', from=2-1, to=3-1]
	\arrow["{\K F(i)}"', from=3-1, to=4-2]
	\arrow["{=}"{description}, draw=none, from=1, to=0]
\end{tikzcd}\].
Equivalently,
\[
\alpha_c^d=\K F(i)\circ \alpha_a^b\circ F(i).
\]
\end{proposition}

\begin{proof}
For each interval \([a,b]\), the object \(\K F_a^b\) is defined by the colimit in Equation~\eqref{eq:K-definition}, and
\[
\alpha_a^b:F_a^b\longrightarrow \K F_a^b
\]
is the corresponding canonical cocone map. If \(i\colon [a,b]\hookrightarrow[c,d] \in \T\), then \(F(i)\colon F_c^d\to F_a^b\) is induced by the functoriality of \(F\in\Pe\). Moreover, \(i\) induces a restriction of the canonical cocone defining \(\K F_c^d\) to a cocone on the diagram defining \(\K F_a^b\). By the universal property of the colimit \(\K F_a^b\), there is a unique morphism
\[
\K F(i)\colon \K F_a^b\longrightarrow \K F_c^d
\]
such that
\[
\alpha_c^d=\K F(i)\circ\alpha_a^b\circ F(i).
\]
Hence,
\[
\bigl(F(i),\K F(i)\bigr)\colon
\bigl(F_c^d\xrightarrow{\alpha_c^d}\K F_c^d\bigr)
\longrightarrow
\bigl(F_a^b\xrightarrow{\alpha_a^b}\K F_a^b\bigr)
\]
is a morphism in \(\cotwD\). Functoriality of \(\cotwK F\colon \T^{op}\to\cotwD\) follows immediately from the functoriality of \(F\), of \(\K F\), and from the uniqueness of the induced maps between the corresponding colimits.

It remains to show that \(\cotwK F\) is a sheaf. Let \([a,b]\in\T\), and consider the cover \(([a,p],[p,b])\). We must show that \(\cotwK F_a^b\) is the following pullback in \(\cotwD\).

% https://q.uiver.app/#q=WzAsNCxbMiwxLCJcXGNvdHd7XFxLfSBYX3BeYiAiXSxbMSwyLCJcXGNvdHd7XFxLfSBYX3BecCJdLFsxLDAsIlxcY290d3tcXEt9IFhfYV5iIl0sWzAsMSwiXFxjb3R3e1xcS30gWF9hXnAiXSxbMCwxLCIoWChrKSxcXEsgWChrKSkiXSxbMiwzLCIoWChmKSxcXEsgWChmKSkiLDJdLFszLDEsIihYKGgpLFxcSyBYKGgpKSIsMl0sWzIsMCwiKFgoZyksXFxLIFgoZykpIl0sWzUsNCwiPSIsMSx7InNob3J0ZW4iOnsic291cmNlIjoyMCwidGFyZ2V0IjoyMH0sInN0eWxlIjp7ImJvZHkiOnsibmFtZSI6Im5vbmUifSwiaGVhZCI6eyJuYW1lIjoibm9uZSJ9fX1dXQ==
\[\begin{tikzcd}
	& {\cotwK  F_a^b} & \\
	{\cotwK  F_a^p} && {\cotwK  F_p^b } \\
	& {\cotwK  F_p^p}
	\arrow[""{name=0, anchor=center, inner sep=0}, "{(F(f),\K F(f))}"', from=1-2, to=2-1]
	\arrow["{(F(g),\K F(g))}", from=1-2, to=2-3]
	\arrow["{(F(h),\K F(h))}"', from=2-1, to=3-2]
	\arrow[""{name=1, anchor=center, inner sep=0}, "{(F(k),\K F(k))}", from=2-3, to=3-2]
	\arrow["{=}"{description}, draw=none, from=0, to=1]
\end{tikzcd}\]

This square is well defined in \(\cotwD\). Indeed, \(F\) is a persistent sheaf, so \(F_a^b\) is the pullback of \(F_a^p\to F_p^p\leftarrow F_p^b\), whereas \(\K F\) is a cumulative narrative, so \(\K F_a^b\) is the pushout of \(\K F_a^p\leftarrow \K F_p^p\to \K F_p^b\).  Let \(u\colon x\to y\) be an arbitrary object of \(\cotwD\), together with compatible morphisms \((p_l,c_l)\colon (x\xrightarrow{u}y)\to\cotwK F_a^p\) and \((p_r,c_r)\colon (x\xrightarrow{u}y)\to\cotwK F_p^b\), as in the following diagram.

% https://q.uiver.app/#q=WzAsNCxbMiwxLCJcXGNvdHd7XFxLfSBYX3BeYiAiXSxbMSwyLCJcXGNvdHd7XFxLfSBYX3BecCJdLFsxLDAsInUiXSxbMCwxLCJcXGNvdHd7XFxLfSBYX2FecCJdLFswLDEsIihYKGspLFxcSyBYKGspKSJdLFszLDEsIihYKGgpLFxcSyBYKGgpKSIsMl0sWzIsMCwiKHBfciwgY19yKSJdLFsyLDMsIihwX2wsIGNfbCkiLDJdLFs3LDQsIj0iLDEseyJzaG9ydGVuIjp7InNvdXJjZSI6MjAsInRhcmdldCI6MjB9LCJzdHlsZSI6eyJib2R5Ijp7Im5hbWUiOiJub25lIn0sImhlYWQiOnsibmFtZSI6Im5vbmUifX19XV0=
\[\begin{tikzcd}
	& u & \\
	{\cotwK  F_a^p} && {\cotwK  F_p^b } \\
	& {\cotwK  F_p^p}
	\arrow[""{name=0, anchor=center, inner sep=0}, "{(p_l, c_l)}"', from=1-2, to=2-1]
	\arrow["{(p_r, c_r)}", from=1-2, to=2-3]
	\arrow["{(F(h),\K F(h))}"', from=2-1, to=3-2]
	\arrow[""{name=1, anchor=center, inner sep=0}, "{(F(k),\K F(k))}", from=2-3, to=3-2]
	\arrow["{=}"{description}, draw=none, from=0, to=1]
\end{tikzcd}\]

Equivalently, the following diagram in \(\D\) commutes.

\[\begin{tikzcd}
	x &&& \\
	\\
	&& {F_a^b} \\
	& {F_a^p} && {F_p^b} \\
	&& {F_p^p} \\
	&& {\K  F_p^p} \\
	& {\K  F_a^p} && {\K  F_p^b } \\
	&& {\K F_a^b} \\
	\\
	y
	\arrow[""{name=0, anchor=center, inner sep=0}, "{p_l}", from=1-1, to=4-2]
	\arrow[""{name=1, anchor=center, inner sep=0}, "{p_r}", curve={height=-30pt}, from=1-1, to=4-4]
	\arrow[""{name=2, anchor=center, inner sep=0}, "u"', from=1-1, to=10-1]
	\arrow["{F(f)}"', from=3-3, to=4-2]
	\arrow["{F(g)}", from=3-3, to=4-4]
	\arrow["\lrcorner"{anchor=center, pos=0.125, rotate=-45}, draw=none, from=3-3, to=5-3]
	\arrow["{=}"{description}, draw=none, from=4-2, to=4-4]
	\arrow["{F(h)}"', from=4-2, to=5-3]
	\arrow[""{name=3, anchor=center, inner sep=0}, "{\alpha_a^p}"', from=4-2, to=7-2]
	\arrow["{F(k)}", from=4-4, to=5-3]
	\arrow[""{name=4, anchor=center, inner sep=0}, "{\alpha_p^b}", from=4-4, to=7-4]
	\arrow[""{name=5, anchor=center, inner sep=0}, "{\alpha_p^p}"', from=5-3, to=6-3]
	\arrow[""{name=6, anchor=center, inner sep=0}, "{\K F(f)}"', from=6-3, to=7-2]
	\arrow["{\K F(g)}", from=6-3, to=7-4]
	\arrow["{\K F(h)}"', from=7-2, to=8-3]
	\arrow[""{name=7, anchor=center, inner sep=0}, "{c_l}"'{pos=0.3}, from=7-2, to=10-1]
	\arrow[""{name=8, anchor=center, inner sep=0}, "{\K F(k)}", from=7-4, to=8-3]
	\arrow[""{name=9, anchor=center, inner sep=0}, "{c_r}", curve={height=-30pt}, from=7-4, to=10-1]
	\arrow["\lrcorner"{anchor=center, pos=0.125, rotate=135}, draw=none, from=8-3, to=6-3]
	\arrow["{=}"{description}, draw=none, from=0, to=1]
	\arrow["{=}"{description}, draw=none, from=2, to=3]
	\arrow["{=}"{description}, draw=none, from=3, to=5]
	\arrow["{=}"{description}, draw=none, from=5, to=4]
	\arrow["{=}"{description}, draw=none, from=6, to=8]
	\arrow["{=}"{description}, draw=none, from=7, to=9]
\end{tikzcd}\]

Since the upper square is a pullback, there is a unique morphism \(p\colon x\to F_a^b\) such that \(F(f)\circ p=p_l\) and \(F(g)\circ p=p_r\). Likewise, since the lower square is a pushout, there is a unique morphism \(c\colon \K F_a^b\to y\) such that \(c\circ\K F(h)=c_l\) and \(c\circ\K F(k)=c_r\). Moreover, the commutativity of the outer diagrams implies that \(u=c\circ\alpha_a^b\circ p\), so \((p,c)\colon (x\xrightarrow{u}y)\to (F_a^b\xrightarrow{\alpha_a^b}\K F_a^b)\) is a morphism in \(\cotwD\). Since \(p\) and \(c\) are uniquely determined by the pullback and pushout universal properties, respectively, \((p,c)\) is the unique morphism making the required diagrams commute. Therefore, \(\cotwK F_a^b\) satisfies the universal property of the pullback in \(\cotwD\), and hence, \(\cotwK F\) is a sheaf.
\end{proof}

\paragraph{The persistent completion.}

Dually, every cumulative narrative admits a canonical persistent completion into a \(\cotwD\)-valued narrative.

\begin{proposition}[Persistent completion of a cumulative narrative]\label{prop:per-completion}
There is a functor
\[
\cotwP:\Cu\to\Nar
\]
that assigns to each cumulative narrative \(\hat F\in\Cu\) the \(\cotwD\)-valued narrative \(\cotwP\hat F\) defined as follows.
\begin{itemize}
    \item On objects \([a,b] \in \T\), the sheaf \(\cotwP \hat{F}\) assigns the morphism
    \[
    \fP \hat{F}_a^b \xrightarrow{\beta_a^b} \hat{F}_a^b,
    \]
    viewed as an object in the cotwisted arrow category \(\cotwD\). The morphism \(\beta_a^b\) is the limit cone morphism from \(\fP \hat{F}_a^b\) to the object \(\hat{F}_a^b\).
    
    \item On morphisms \([a,b] \xhookrightarrow[]{i} [c,d]\) in \(\T\), the sheaf \(\cotwP \hat{F}\) assigns the morphism in \(\cotwD\)
    \[
    (\fP \hat{F}_c^d \xrightarrow{\fP \hat{F}(i)} \fP \hat{F}_a^b, \hat{F}_a^b \xrightarrow{\hat{F}(i)} \hat{F}_c^d),
    \]
    which satisfies the cotwisted commutativity condition

% https://q.uiver.app/#q=WzAsNCxbMSwwLCJcXGZQIFxcaGF0e1h9X2NeZCAiXSxbMSwzLCJcXGhhdHtYfV9jXmQiXSxbMCwxLCJcXGZQIFxcaGF0e1h9X2FeYiJdLFswLDIsIlxcaGF0e1h9X2FeYiJdLFswLDEsIlxcYmV0YV9jXmQgIl0sWzAsMiwiXFxmUCBcXGhhdHtYfShpKSIsMl0sWzMsMSwiXFxoYXR7WH0oaSkiLDJdLFsyLDMsIlxcYmV0YV9hXmIiLDJdLFs3LDQsIj0iLDEseyJzaG9ydGVuIjp7InNvdXJjZSI6MjAsInRhcmdldCI6MjB9LCJzdHlsZSI6eyJib2R5Ijp7Im5hbWUiOiJub25lIn0sImhlYWQiOnsibmFtZSI6Im5vbmUifX19XV0=
\[\begin{tikzcd}
	& {\fP \hat{F}_c^d } \\
	{\fP \hat{F}_a^b} \\
	{\hat{F}_a^b} \\
	& {\hat{F}_c^d}
	\arrow["{\fP \hat{F}(i)}"', from=1-2, to=2-1]
	\arrow[""{name=0, anchor=center, inner sep=0}, "{\beta_c^d }", from=1-2, to=4-2]
	\arrow[""{name=1, anchor=center, inner sep=0}, "{\beta_a^b}"', from=2-1, to=3-1]
	\arrow["{\hat{F}(i)}"', from=3-1, to=4-2]
	\arrow["{=}"{description}, draw=none, from=1, to=0]
\end{tikzcd}\].
Explicitly,
\[
    \beta_c^d =  \hat{F}(i)  \circ \beta_a^b  \circ \fP \hat{F}(i).
\]
\end{itemize}
\end{proposition}

\begin{proof}
The construction is the categorical dual of Proposition~\ref{prop:cum-completion}.  For each interval \([a,b]\), the object \(\fP\hat{F}_a^b\) is defined by the limit in Equation~\eqref{eq:P-definition}, with canonical cone morphism \(\beta_a^b\colon\fP\hat{F}_a^b\to\hat{F}_a^b\). Given an inclusion \(i\colon[a,b]\hookrightarrow[c,d]\), the universal property of the limit induces a unique morphism \(\fP\hat{F}(i)\colon\fP\hat{F}_c^d\to\fP\hat{F}_a^b\) satisfying \(\beta_c^d=\hat{F}(i)\circ\beta_a^b\circ\fP\hat{F}(i)\). Hence \((\fP\hat{F}(i),\hat{F}(i))\) is a morphism in \(\cotwD\), and these assignments define the functor \(\cotwP \hat{F}\).

To verify the sheaf axiom, let \(([a,p],[p,b])\) be a cover of \([a,b]\). Given an arbitrary object \(u\colon x\to y\) of \(\cotwD\) together with compatible morphisms to \(\cotwP\hat{F}_a^p\) and \(\cotwP\hat{F}_p^b\), since \(\hat{F}\) is a cumulative narrative, \(\hat{F}_a^b\) is the pushout of \(\hat{F}_a^p\leftarrow\hat{F}_p^p\rightarrow\hat{F}_p^b\), while \(\fP\hat{F}\) is a persistent narrative, so \(\fP\hat{F}_a^b\) is the pullback of \(\fP\hat{F}_a^p\rightarrow\fP\hat{F}_p^p\leftarrow\fP\hat{F}_p^b\). The universal properties of the pushout \(\hat{F}_a^b\) and the pullback \(\fP\hat{F}_a^b\) yield unique morphisms \(p\colon x\to\fP\hat{F}_a^b\) and \(c\colon\hat{F}_a^b\to y\), which satisfy the cotwisted compatibility condition and therefore determine the unique morphism \((p,c)\) in \(\cotwD\). The verification of this compatibility is exactly dual to that in Proposition~\ref{prop:cum-completion}, and is therefore omitted.
\end{proof}

\begin{proof}[Proof of Theorem~\ref{thm:adjunction-factorization}]
The functors involved are well defined by Proposition~\ref{prop:dom-cod-functors}, Proposition~\ref{prop:cum-completion}, and Proposition~\ref{prop:per-completion}. It remains to verify the  identities
\[
\cotwcod\circ\cotwK=\K,
\qquad
\cotwdom\circ\cotwP=\fP,
\]
\[
\cotwdom\circ\cotwK=\id_{\Pe},
\qquad
\cotwcod\circ\cotwP=\id_{\Cu}.
\]

For \(F\in\Pe\), we have \(\cotwK F_a^b=(F_a^b\xrightarrow{\alpha_a^b}\K F_a^b)\) and, for every inclusion \(i:[a,b]\hookrightarrow[c,d]\), \(\cotwK F(i)=\bigl(F(i),\K F(i)\bigr)\). Hence,
\[
\cotwcod(\cotwK F)=\K F,
\qquad
\cotwdom(\cotwK F)=F,
\]
so \(\cotwcod\circ\cotwK=\K\) and \(\cotwdom\circ\cotwK=\id_{\Pe}\).

Similarly, for \(\hat F\in\Cu\), we have \(\cotwP\hat F_a^b=(\fP\hat F_a^b\xrightarrow{\beta_a^b}\hat F_a^b)\) and, for every inclusion \(i:[a,b]\hookrightarrow[c,d]\), \(\cotwP\hat F(i)=\bigl(\fP\hat F(i),\hat F(i)\bigr)\). Hence,
\[
\cotwdom(\cotwP\hat F)=\fP\hat F,
\qquad
\cotwcod(\cotwP\hat F)=\hat F,
\]
so \(\cotwdom\circ\cotwP=\fP\) and \(\cotwcod\circ\cotwP=\id_{\Cu}\).
\end{proof}

\subsubsection{Classifying narratives by rigidity}\label{subsubsec:rigidity}

Since \(X\in\Nar\) is valued in the cotwisted arrow category, it
canonically determines, for every interval \([a,b]\), a comparison morphism
\[
\gamma_a^b:P_a^b\longrightarrow C_a^b,
\]
where
\[
X_a^b=
\left(
P_a^b\xrightarrow{\gamma_a^b}C_a^b
\right).
\]
Applying the completion functors to the persistent and cumulative components of \(X\) produces the narratives
\[
\cotwK(\cotwdom(X))
\qquad\text{and}\qquad
\cotwP(\cotwcod(X)).
\]
By construction of the functors \(\cotwK\) and \(\cotwP\), the structure morphism of \(\cotwK(\cotwdom(X))\) at an interval \([a,b]\) is the component
\[
\bigl(\eta_{\cotwdom(X)}\bigr)_a^b:
\cotwdom(X)_a^b\longrightarrow
(\fP\K\cotwdom(X))_a^b
\]
of the unit of the adjunction, while the structure morphism of \(\cotwP(\cotwcod(X))\) is the component
\[
\bigl(\varepsilon_{\cotwcod(X)}\bigr)_a^b:
(\K\fP\cotwcod(X))_a^b
\longrightarrow
\cotwcod(X)_a^b
\]
of the counit. These assemble into natural transformations
\[
\eta_{\cotwdom(X)}:
\cotwdom(X)\Longrightarrow
\fP\K(\cotwdom(X))
\]
and
\[
\varepsilon_{\cotwcod(X)}:
\K\fP(\cotwcod(X))
\Longrightarrow
\cotwcod(X),
\]
which we call the \emph{canonical comparison morphisms} associated to the narrative \(X\). These comparison morphisms measure how closely the persistent and cumulative descriptions encoded by a narrative agree with the canonical ones determined by the persistence--accumulation adjunction. This observation motivates the following rigidity classification of narratives.

\begin{definition}[Left rigid narrative]
A narrative \(X\in\Nar\) is called \emph{left rigid} if the comparison morphism
\[
\eta_{\cotwdom(X)}:
\cotwdom(X)
\longrightarrow
\fP\K(\cotwdom(X))
\]
is an isomorphism.
\end{definition}

Equivalently, the persistent component of \(X\) is completely recovered after passing to its canonical cumulative completion and back.

\begin{example}[Spans of monomorphisms in adhesive categories]
Assume that the ambient category \(\D\) is adhesive, and let \(X\in\Nar\) be a narrative whose persistent component
\[
\cotwdom(X)_0^0
\xleftarrow{}
\cotwdom(X)_0^1
\xrightarrow{}
\cotwdom(X)_1^1
\]
consists of monomorphisms. Since pushouts along monomorphisms are Van Kampen~\cite{LackAdhesive}, the associated pushout square is also a pullback. Hence, \(\fP\K(\cotwdom(X))\cong\cotwdom(X)\), and therefore, \(X\) is left rigid.

This class of examples is particularly relevant in double-pushout graph rewriting, where rewriting rules are represented by spans of monomorphisms in adhesive categories~\cite{LackAdhesive,fundamentalsOfAlgebraicGraphTransformation}. Consequently, for this broad class of graph transformations, the passage from persistent to cumulative descriptions and back preserves the original persistent description.
\end{example}

\begin{remark}
It is worth emphasizing that the converse need not hold. Although adhesive categories guarantee that spans of monomorphisms are preserved by the pushout--pullback composite, they do not in general imply that cospans are preserved by the pullback--pushout composite. Thus, even in adhesive categories, left rigidity does not automatically imply right rigidity.
\end{remark}

\begin{definition}[Right rigid narrative]
A narrative \(X\in\Nar\) is called \emph{right rigid} if the comparison morphism
\[
\varepsilon_{\cotwcod(X)}:
\K\fP(\cotwcod(X))
\longrightarrow
\cotwcod(X)
\]
is an isomorphism.
\end{definition}

Equivalently, the cumulative component of \(X\) is completely recovered after passing to its canonical persistent completion and back.

\begin{example}[A right rigid narrative]
Consider the narrative \(X\) given by:
\[
\begin{tikzcd}[column sep=large,row sep=large]
& \{0,1\} \arrow[dl,"p_\ell"'] \arrow[dr,"p_r"] & \\
\{a\} \arrow[dr,"c_\ell"'] && \{b\} \arrow[dl,"c_r"] \\
& \{*\}. &
\end{tikzcd}
\]
Applying \(\fP\) to the cumulative component \(\cotwcod(X)\) computes the pullback \(\{a\}\times_{\{*\}}\{b\}\cong\{(a,b)\}.\)
Applying \(\K\) to the persistent narrative
\(\{a\}\xleftarrow{\pi_\ell}\{(a,b)\}\xrightarrow{\pi_r}\{b\} =\fP(\cotwcod(X))\)
computes its pushout, which is again the singleton \(\{*\}\). Therefore, the counit
\(\varepsilon_{\cotwcod(X)}: \K\fP(\cotwcod(X)) \longrightarrow \cotwcod(X)\) is an isomorphism. Hence, \(X\) is right rigid. Moreover, \(X\) is not left rigid. Indeed, \(\fP\K(\cotwdom(X))\) has apex \(\{(a,b)\}\), while \(\cotwdom(X)\) has apex \(\{0,1\}\). Hence the unit \(\eta_{\cotwdom(X)}:\cotwdom(X)\longrightarrow\fP\K(\cotwdom(X))\) is not an isomorphism.
\end{example}

\begin{definition}[Rigid and loose narratives]
A narrative is called
\begin{itemize}
    \item \emph{rigid} if it is both left rigid and right rigid;
    \item \emph{loose} if it is neither left rigid nor right rigid.
\end{itemize}
\end{definition}

\begin{example}[Rigid narratives]
The simplest examples of rigid narratives are the constant ones. Let \(d\in\D\). The \emph{constant narrative} is the \(\cotwD\)-valued sheaf \(X\) defined by \(X_a^b=(d\xrightarrow{\id_d}d)\) for every interval \([a,b]\in\T\), with every restriction morphism equal to \((\id_d,\id_d)\). Since both the persistent and cumulative components are constant, one has
\[
\cotwdom(X)=\fP(\cotwcod(X))
\qquad\text{and}\qquad
\cotwcod(X)=\K(\cotwdom(X)).
\]
Hence \(X\) is rigid.  More generally, the same conclusion holds whenever the morphisms of the narrative are isomorphisms. Indeed, replacing the identities above by arbitrary isomorphisms does not change either the pullback or pushout constructions up to canonical isomorphism, so both comparison morphisms remain isomorphisms. Consider, for instance:
\[
\begin{tikzcd}[column sep=large,row sep=large]
& \{(a,b)\} \arrow[dl,"\pi_\ell"'] \arrow[dr,"\pi_r"] & \\
\{a\} \arrow[dr,"c_\ell"'] && \{b\} \arrow[dl,"c_r"] \\
& \{*\}. &
\end{tikzcd}
\]
The persistent component \(\{a\}\xleftarrow{\pi_\ell}\{(a,b)\}\xrightarrow{\pi_r}\{b\}\) is already the pullback of the cospan \(\{a\}\xrightarrow{c_\ell}\{*\}\xleftarrow{c_r}\{b\}\), while the pushout of this pullback is again the singleton \(\{*\}\).
\end{example}

\begin{example}[Loose narratives]
Loose narratives already appear in very simple categories. Consider the poset category \((\mathbb N,\leq)\), whose objects are natural numbers and in which there exists a unique morphism \(m\to n\) precisely when \(m\leq n\). In this category, pullbacks are given by meets \(m\times_p n=\min\{m,n\}\), while pushouts are given by joins \(m+_p n=\max\{m,n\}\).

Consider the narrative
% https://q.uiver.app/#q=WzAsNCxbMSwwLCIwIl0sWzAsMSwiMSJdLFsxLDIsIjMiXSxbMiwxLCIyIl0sWzAsMV0sWzMsMl0sWzAsM10sWzEsMl1d
\[\begin{tikzcd}
	& 0 & \\
	1 && 2 \\
	& 3
	\arrow[from=1-2, to=2-1]
	\arrow[from=1-2, to=2-3]
	\arrow[from=2-1, to=3-2]
	\arrow[from=2-3, to=3-2]
\end{tikzcd}\]
Its domain is the span \(\cotwdom(X)=\left(1\leftarrow 0\rightarrow 2\right)\), and applying the completion functors yields \(\K(\cotwdom(X))=\left(1\rightarrow 2\leftarrow 2\right)\) and \(\fP\K(\cotwdom(X))=\left(1\leftarrow 1\rightarrow 2\right)\), which is not isomorphic to \(\cotwdom(X)\). Thus, \(X\) is not left rigid. Likewise, its codomain is the cospan \(\cotwcod(X)=\left(1\rightarrow 3\leftarrow 2\right)\), and applying the completion functors gives \(\fP(\cotwcod(X))=\left(1\leftarrow 1\rightarrow 2\right)\) and \(\K\fP(\cotwcod(X))=\left(1\rightarrow 2\leftarrow 2\right)\), which is not isomorphic to \(\cotwcod(X)\). Hence, \(X\) is not right rigid, and therefore it is loose.
\end{example}

The central question explored in this vignette is when persistent and cumulative descriptions determine one another exactly. Investigating the fixed points of the persistence--accumulation adjunction naturally leads to the category of narratives, which simultaneously records both descriptions together with the comparison between them. This approach provides a specialized framework for systematically investigating the conditions—relating to data category and temporal data assignments—under which information is preserved as one transitions between persistent and cumulative representations.

\subsection{Vignette 2: Decomposing Time-Varying Data into Simple Pieces: Structured Decompositions of Narratives}\label{subsec:decompositions}

Complex data, whether static or time-varying, is often easier to understand and analyze when it is represented as a composite of smaller or simpler pieces. Thus, in this section, we ask: \textit{how does one decompose complicated time-varying data into simpler pieces in a systematic way?} A decomposition serves two roles: (1) it divides an object into component pieces, and (2) it records how these pieces overlap, so that the original object may be reconstructed by gluing them together. Thinking about this with a more topological slant, there is a sense in which this section is about seeking simple coverings of a temporal object by small, time-varying pieces, analogous to open sets, whose evolution and mutual compatibility are themselves tracked through time. 

Our goal is to obtain an invariant of the structural complexity of a time-varying object by measuring the complexity of the pieces in its decompositions and their overlaps. These pieces must themselves form time-varying objects, with structure maps recording how they persist, merge, split, or disappear. It is not enough to decompose each snapshot independently: the pieces of the decomposition must also capture the evolution of the global structure. We seek a method that reuses theories of decomposition for static data instead of defining a new notion for each temporal setting. This section is based on the recent paper by Bumpus and Nickel~\cite{decomposingtimevaryingdata2026} whose main result provides such a method by lifting theories of decompositions from static objects to persistent narratives.

\newcommand{\Grefl}{\textup{\textsf{Grph}}_\textup{\textsf{refl}}}

\subsubsection{Structured decompositions}
\label{subsubsection:structured-decompositions}

In this section, we use narratives to construct a temporalized theory of structured decompositions, which were introduced in~\cite{Bumpus-et-al_Structured-Decompositions} and provide a category theoretical generalization of tree-decompositions. The goal of our approach is to split time-varying data systematically into smaller components, providing a formal framework for dealing with temporal systems. This is an instance of the broader principle of compositionality, which states that the meaning or behaviour of a complex system is completely determined by the meanings or behaviours of its constituent parts and the rules governing their connections.

Breaking complicated data into simpler pieces has already proved useful in many areas of mathematics and computer science, particularly in graph theory, logic, algorithms, and complexity theory. Prominent examples include Robertson and Seymour's graph structure theorem~\cite{Robertson-Seymour_Graph-minors-XVII} and Courcelle's theorem~\cite{Courcelle}.

Here, a \emph{graph} $G$ means a diagram from the category $V\rightrightarrows E$ to the category $\Set$ of sets. Thus, a graph in this sense is the same as a \emph{quiver} in representation theory and as a copresheaf on $V\rightrightarrows E$ in category theory.
In particular, our graphs are directed and are allowed to possess loops and multiple edges. Given a graph~$G$, we denote the set of its \emph{vertices} by $V(G)\coloneqq G(V)$ and the set of its \emph{edges} by $E(G)\coloneqq G(E)$.
We write $\Grph$ for the wide subcategory of the functor category $[ V\rightrightarrows E, \Set]$ whose objects are the graphs.
Its morphisms are simply natural transformations between graphs, also called \emph{graph morphisms}.

\paragraph{Tree-decompositions.}
Originally introduced by Halin in 1976~\cite{Halin}, the notion of a \textit{tree-decomposition} (see Definition~\ref{defi:tree-decomposition}) was rediscovered in 1984 by Robertson and Seymour~\cite{Robertson-Seymour_Graph-minors-III} and has since played an essential role in structural and extremal graph theory, as well as in various other areas of mathematics and computer science. For instance, tree-decompositions have proved to be a convenient tool in matrix decomposition~\cite{Liu}, query optimization~\cite{Yannakakis}, and dynamic programming~\cite{Arnborg-Proskurowski}. They are also frequently used to solve constraint satisfaction problems~\cite{Dechter-Pearl} and in junction tree algorithms for probabilistic inference~\cite{Lauritzen-Spiegelhalter}. Robertson and Seymour themselves used tree-decompositions as a crucial component of their celebrated graph structure theorem~\cite{Robertson-Seymour_Graph-minors-XVII}, which establishes a profound connection between graph minor theory and the theory of topological embeddings. Its significance is further illustrated by applications to the disjoint paths problem~\cite{Robertson-Seymour_Graph-minors-XIII}, Sachs' linkless embedding conjecture~\cite{Robertson-Seymour-Thomas}, and Courcelle's theorem~\cite{Courcelle}.

Moreover, tree-decompositions have proved very useful for investigating compositional structures. They serve as an efficient topological tool for tackling algorithmic graph problems and have been used to improve the efficiency of dynamic programming. For example, with the aid of tree-decompositions, several algorithmic problems that are NP-hard on arbitrary graphs may be solved efficiently by dynamic programming on graphs of bounded tree-width. A concrete example is the problem of finding a maximum independent set in graphs of bounded tree-width.

To acknowledge the benefit of structured decompositions (see Definition~\ref{defi:structured-decompositions}) we make the notion of ordinary tree-decompositions precise.   

\begin{definition}\label{defi:tree-decomposition}
	A \textit{tree-decomposition} of a graph~$G$ is a pair $(T, (V_t)_{t\in\Ob(T)})$ consisting of a tree~$T$ and a family of vertex sets $V_t\subseteq V(G)$ indexed by the vertices $t$ of~$T$ and satisfying the following two conditions:
	\begin{enumerate}[label=(T\theenumi), leftmargin=35pt]
		\item $G=\bigcup_{t\in\Ob(T)}G[V_t]$ where $G[V_t]$ denotes the subgraph of~$G$ induced by~$V_t$.
		\item For each vertex $v$ of~$G$, the induced subgraph $T_v := T[\{t\in T\mathbin| v\in V_t\}]\subseteq T$ is connected.
	\end{enumerate}
	We call $T$ the \textit{decomposition tree} or the \textit{model} of the decomposition. The subsets $V_t\subseteq V(G)$ are called the \textit{bags}, and the induced subgraphs $G[V_t]\subseteq G$ are the corresponding \textit{parts}.
\end{definition}

Intuitively, a tree-decomposition reveals the global structure of a graph whenever that structure is tree-like. An example of a tree-decomposition is shown in Figure~\ref{fig:tree-decomposition}.

\begin{figure}[H]
	\centering
	\includegraphics[width=0.8\textwidth]{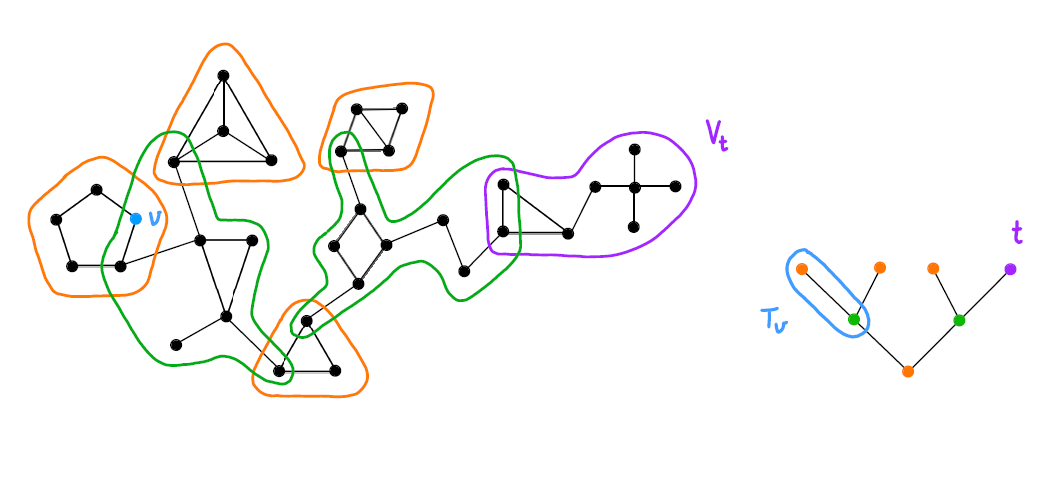}
	\setlength{\abovecaptionskip}{-1cm}
	\caption{A tree-decomposition of a graph (on the left-hand side) and its decomposition tree (on the right-hand side).}
	\label{fig:tree-decomposition}
\end{figure}

\paragraph{Tree-width.}
\label{paragraph:tree-width}

Tree-decompositions also provide a convenient tool for constructively computing the \textit{tree-width} of a graph, as explained below. Originally introduced by Bertelè and Brioschi, the notion of tree-width has become indispensable in mathematics and computer science. For instance,~\cite{Fujita_Width} provides a revealing overview of major applications of tree-width and more general notions of graph width that have been developed from the classical tree-width parameter to capture a broader range of contexts.

The tree-width of a graph~$G$ measures the extent to which the structure of~$G$ resembles that of a tree. Indeed, the tree-like structure of the graph is reflected in the parts of its decomposition: the smaller these parts are, the more tree-like~$G$ is.

\begin{definition}\label{defi:width-tree-width}
	The \textit{width} of a tree-decomposition $(T, (V_t)_{t\in\Ob(T)})$ is the maximum order of its bags minus one,
	$\max_{t\in V(T)} \abs{V_t}-1$.
	The \textit{tree-width} $\textup{tw}(G)$ of a graph~$G$ is the minimum width of its tree-decompositions,
	\[ \textup{tw}(G) \coloneqq \min_{(T,(V_t)_{t\in\Ob(T)})} \, \max_{t\in V(T)} \abs{V_t}-1 \]
	where the minimum is taken over all tree-decompositions of~$G$.
\end{definition}
The reason for the ``minus one'' in the definition of the width of a tree-decomposition is to ensure that the tree-width of any tree is one, as expected from a reasonable measure of structural tree-likeness.

\paragraph{Category theoretical generalization.}
To provide a general category-theoretical framework for tree-decompositions and extend their scope of application, the notion of structured decompositions and, more specifically, spined structured decomposition categories (or simply \emph{spined sd-categories}) was introduced in~\cite{Bumpus-et-al_Structured-Decompositions}. These provide a unified axiomatic setting for studying tree-width and its generalizations. The framework recovers several notions of graph width, including ordinary tree-width, complemented tree-width, tree independence number, hypergraph tree-width, and layered tree-width. Thus, spined sd-categories capture the classical notion of tree-width together with many of its variants. We define them explicitly in Definition~\ref{defi:structured-decompositions}.

\subsubsection{Temporalization of spined sd-categories}

\paragraph{Need for temporalization.}
Structured decompositions are concerned only with static graphs, meaning graphs in the traditional sense, consisting merely of a set of vertices together with a set of edges connecting them. When applying the existing concepts and results to time-varying data, however, this conventional notion of a graph is no longer sufficient. Many systems of interest evolve over time, and traditional graph-theoretical techniques ignore this temporal dimension. To capture such evolution, a variety of notions of temporal graphs have been introduced. These differ in how they represent temporal information and the evolution of the underlying graph. For an overview of the diversity of existing approaches, we refer to~\cite{Harary-Gupta, Kempe-et-al, Casteigts-et-al, Holme-Saramaeki, Holme, Michail}.

Motivated by the increasing use of temporal graphs to model time-dependent data, we develop a notion of spined sd-categories for time-varying graphs and, more generally, time-varying structures. This is the main objective of the remainder of this section, which is based on~\cite{decomposingtimevaryingdata2026}.

\paragraph{Combining spined sd-categories and narratives.}
To temporalize spined sd-categories, we reinterpret the theory of structured decompositions developed in~\cite{Bumpus-et-al_Structured-Decompositions} within the framework of persistent and cumulative narratives developed in~\cite{theorytimevaryingdata2026}. This yields a time-dependent generalization of spined sd-categories.

\paragraph{Structured decompositions.}
\begin{definition}
Let \(J\) be a graph. Its \emph{barycentric subdivision} is the category \(\int J\) obtained from \(J\) by replacing each vertex with an object and each edge \(e\), with endpoints \(v\) and \(w\), by an object \(e\) together with morphisms \(e_v:e\to v\) and \(e_w:e\to w\), that is, a span \(v\xleftarrow{e_v}e\xrightarrow{e_w}w\).
\end{definition}

\begin{example}
	The barycentric subdivision of the triangle $K^3$ depicted on the left is the category visualized on the right.
	\begin{equation*}
		\begin{tikzcd}[row sep=large]
			& v \arrow[ld, "a", no head, swap] \arrow[rd, "b", no head]
			\\ u \arrow[rr, "c", no head, swap] & & w
		\end{tikzcd}
		\qquad \qquad
		\begin{tikzcd}[row sep=large, column sep=huge]
			a \arrow[d, "a_u", swap] \arrow[rd, "a_v", pos=0.8] & c \arrow[ld, "c_u", pos=0.2, swap] \arrow[rd, "c_w", pos=0.2] & b \arrow[ld, "b_v", pos=0.8, swap] \arrow[d, "b_w"] 
			\\ u & v & w
		\end{tikzcd}
	\end{equation*}
\end{example}

\begin{definition}\label{defi:structured-decompositions}
	Let $J$ be a graph, and let $\D$ be a category.
	A \textit{$J$-structured decomposition in~$\D$} is a functor of shape $d\colon\int J\to\D$ with the property that for every morphism $k\colon\thinspace x\to y$ in~$\int J$, its image $d(k)\colon\thinspace d(x)\to d(y)$ under~$d$ is a monomorphism in~$\D$.
\end{definition}

\paragraph{Relevance of structured decompositions and schematic visualizations.}
We illustrate the breadth of structured decompositions by discussing four examples from different mathematical domains, accompanied by schematic visualizations.

\begin{enumerate}
	\renewcommand{\labelenumi}{(\theenumi)}
\item Tree-decompositions. As already mentioned, the development of structured decompositions was originally inspired by tree-decompositions in graph theory, making them perhaps the most prominent example. These combinatorial objects can be described as tree-shaped structured decompositions with values in the category~$\Grph$ of graphs, that is, functors of the form \(\int T\to\Grph\) for some tree~\(T\). Replacing \(T\) with an arbitrary graph leads to the more general notion of a graph-decomposition, studied, for example, in~\cite{Carmesin-et-al_Graph-decompositions,Diestel-et-al_Graph-decompositions}.

The notion of graph-decompositions moreover gives rise to a wide variety of combinatorial width parameters measuring the structural resemblance of a graph to a given graph model, such as a tree, a path, or a cycle. Examples include the classical tree-width together with many of its variants, as discussed in~\cite[Chapter~3]{Bumpus-et-al_Structured-Decompositions}.

Figure~\ref{fig:labelled-example-graphs} shows a structured decomposition in the category~$\Grph$ of graphs shaped by the cycle~\(C^5\) of length five, together with the cycle itself and its barycentric subdivision~\(\int C^5\). The functor \(d\) sends each object \(x_i\) of~\(\int C^5\) corresponding to a vertex \(x_i\) of the graph~\(C^5\) to a complete graph.
    \begin{figure}
        \centering
        \includegraphics[width=0.8\linewidth]{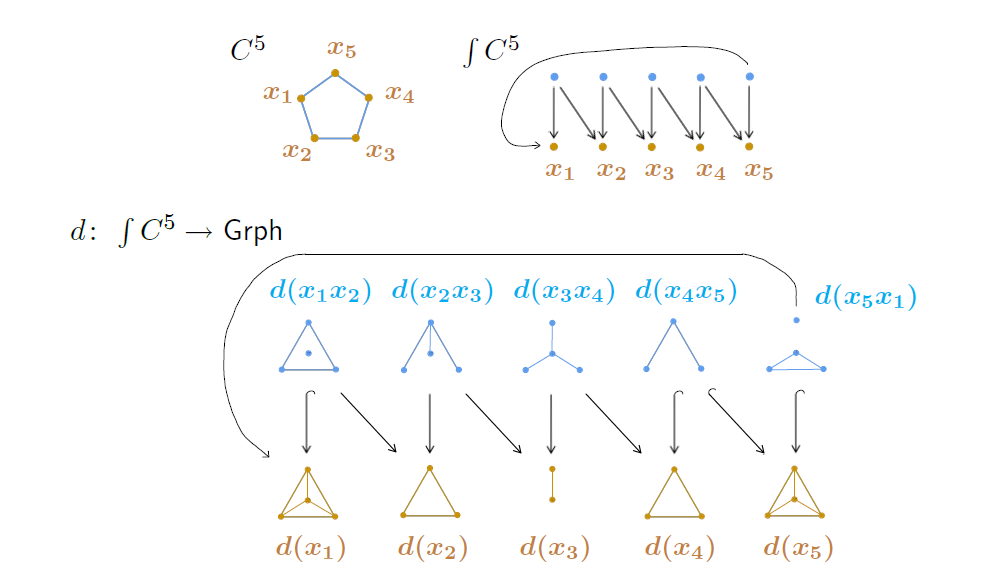}
        \caption{The cycle~$C^5$, the category~$\int C^5$ and a $C^5$-shaped structured decomposition of graphs.}
        \label{fig:labelled-example-graphs}
    \end{figure}

    \item Hybrid dynamical systems.
In his thesis~\cite{Ames_Hybrid-Systems}, Ames establishes a category-theoretical framework for the study of hybrid dynamical systems, that is, dynamical systems with both continuous and discrete components. The thesis introduces the notion of a hybrid object in a category~\(\cat{C}\), defined as a functor from a so-called D-category to~\(\cat{C}\), which may be viewed as a special case of a structured decomposition.

To illustrate this concept, consider a structured decomposition in the category~\(\cat{Man}\) of topological manifolds and continuous maps, shaped by the path~\(P^3\) of length three, as shown in Figure~\ref{fig:labelled-example-manifolds}. We observe that the map \(d(f_1)\colon S^1\to\Sigma_1\), from the topological unit circle~\(S^1\) to the torus~\(\Sigma_1\) (the closed orientable surface of genus one), is homotopic to a constant map. Likewise, the map \(d(f_2)\colon S^1\to\Sigma_2\), where \(\Sigma_2\) denotes the closed orientable surface of genus two, is homotopic to the composite \(S^1\xrightarrow{d(g_3)}\Sigma_1\hookrightarrow\Sigma_2\).
    \begin{figure}
        \centering
        \includegraphics[width=0.8\linewidth]{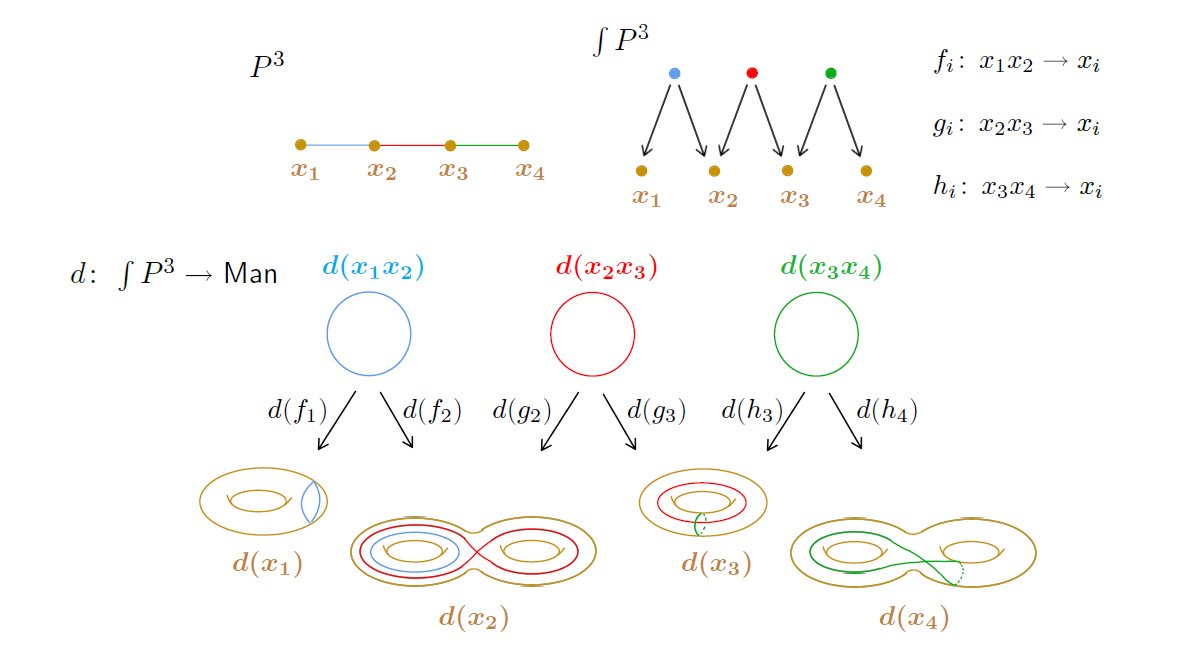}
        \caption{The path~$P^3$, the category~$\int P^3$ and a $P^3$-shaped structured decomposition of manifolds.}
		\label{fig:labelled-example-manifolds}
    \end{figure}

    \item Graphs of groups. The fundamental objects of Bass--Serre theory, namely graphs of groups, are precisely structured decompositions valued in the category of groups and group homomorphisms~\cite{Serre_Arbres,Serre_Trees,Bass}. Restricting to structured decompositions shaped by trees recovers the setting of Bass--Serre theory, in which such decompositions correspond to well-behaved group actions on trees.

As an example, we consider the fundamental groups of the manifolds in the image of the functor \(d\colon\int P^3\to\cat{Man}\) from the previous example, thereby obtaining a \(P^3\)-shaped structured decomposition in the category~\(\cat{Grp}\) of groups, shown in Figure~\ref{fig:labelled-example-groups}. More precisely, this is the composition of the previous structured decomposition in~\(\cat{Man}\) with the fundamental group functor \(\pi_1\colon\cat{Man}\to\cat{Grp}\). In the upper part of the figure, we illustrate closed curves in \(S^1\), \(\Sigma_1\), and~\(\Sigma_2\) generating the corresponding fundamental groups, while the lower part visualizes the resulting structured decomposition of fundamental groups.

As expected, the homotopic maps \(d(f_1)\simeq\mathrm{const}\colon S^1\to\Sigma_1\) induce the same group homomorphism \(\pi_1(S^1)\to\pi_1(\Sigma_1)\), namely the trivial homomorphism. Likewise, the homotopic maps \(d(f_2)\simeq\mathrm{incl}\circ d(g_3)\colon S^1\to\Sigma_2\) induce the same homomorphism \(\pi_1(S^1)\to\pi_1(\Sigma_2)\), namely \(x\mapsto a_1\).
    \begin{figure}
        \centering
        \includegraphics[width=0.8\linewidth]{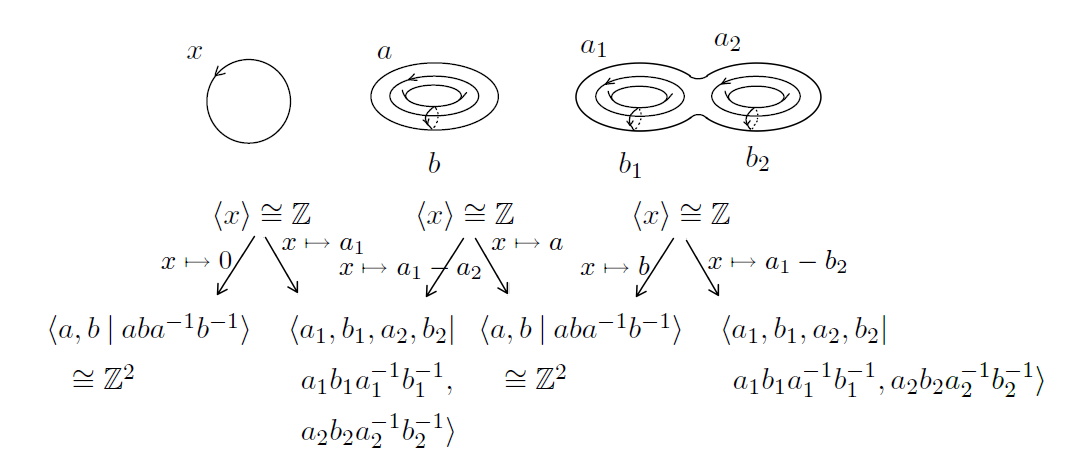}
        \caption{A \(P^3\)-shaped structured decomposition of groups (lower part),
obtained by computing the fundamental groups of the manifolds in Figure~\ref{fig:labelled-example-manifolds} with the chosen generators
(upper part).}
		\label{fig:labelled-example-groups}
    \end{figure}

    \item Cellular sheaves. Cellular sheaves encode local-to-global interactions in systems built from graphs, simplicial complexes, and cell complexes~\cite{Hu}. Later in this chapter, we use them to model multi-agent systems with time-varying communication topologies, illustrating an application of narratives to control theory developed in~\cite{temporal-cellular-sheaves2026}. From the perspective developed here, cellular sheaves arise naturally as the dual notion to structured decompositions: they are structured co-decompositions valued in the category~\(\Vect\) of real vector spaces and linear maps. We nevertheless retain the terminology of cellular sheaf theory in order to align with the existing literature on control theory and multi-agent systems.

\end{enumerate}

\paragraph{Spined sd-categories.}
\begin{definition}
	Let $\D$ be a category, and let $\cG$ be a class of graphs including the trivial graph with exactly one vertex and no edge. The pair $(\D,\cG)$ is said to be a \textit{structured decomposition category}, or simply, an \textit{sd-category}, iff every structured decomposition $d\colon\thinspace\int J\to\D$ in~$\D$ with $J\in\cG$ admits a colimit.
	We call $\cG$ the \textit{index class} of $(\D,\cG)$.
\end{definition}

\begin{definition}
	A \textit{spine} on a category~$\D$ is a family $\Omega = (\Omega_n)_{n\geq0}$ of increasing subcategories 
		\[ \Omega_0\subseteq\Omega_1\subseteq\Omega_2\subseteq\dots \]
	with the following additional properties:
	\begin{itemize}
		\item 
		The union $\bigcup\Omega\coloneqq \bigcup_{n\geq0}\Omega_n$ is closed under isomorphic objects in~$\D$, and for any $n\geq0$, $\Omega_n$~is closed under isomorphic objects as a subcategory of~$\bigcup\Omega$.
		\item 
		For every object $X\in\Ob(\D)$, there exists an object $W\in\Ob(\Omega_n)$ for some integer $n\geq0$ together with a monomorphism $W\rightarrowtail X$ in~$\D$.
		\item 
		Given an integer $n\geq1$ and objects $W$ of~$\Omega_n$ and $\widetilde{W}$ of~$\Omega_{n-1}$ along with a monomorphism $W\rightarrowtail \widetilde{W}$ in~$\D$, then $W$~is already contained in~$\Omega_{n-1}$.
	\end{itemize}	
	A \textit{spined sd-category} is a triple $(\D,\cG,\Omega)$ where $(\D,\cG)$ is an sd-category and $\Omega$ is a spine on~$\D$.		
\end{definition}

\paragraph{Combining persistent narratives and spined sd-categories.}
\label{paragraph:combining-persistent-narratives-spined-sd-cats}
We begin with the persistent perspective; the cumulative perspective and the relation between the two are discussed in Section~\ref{section:future}. Let \(\cat{T}\) be a finite discrete time category, let \(\cat{S}\subseteq\cat{T}\) be a sub-join-semilattice, and let \(\tau:\cat{S}\hookrightarrow\cat{T}\) denote the inclusion functor. Furthermore, let \((\D,\cG,\Omega)\) be a spined sd-category satisfying the following additional condition.
\begin{enumerate}[leftmargin=0.5in]
	\renewcommand{\labelenumi}{(T\theenumi)}
	\item\label{axiom:T1} The category $\D$ and its subcategories $\Omega_n\subseteq\D$ admit all pullbacks, the inclusion functors $\iota_n\colon\thinspace \Omega_n\hookrightarrow\D$ preserve pullbacks, and each $\Omega_n$ is full in~$\D$.
\end{enumerate}
In particular, post-composition with~$\iota_n$ provides a well-defined functor
\[ \cat{Pe}(\cat{T},\iota_n)\colon\thinspace \cat{Pe}(\cat{T},\Omega_n)\to \Pe, \]
which we call the \textit{covariant change-of-base functor}.
Similarly, we have a \textit{change-of-temporal-resolution functor} defined by pre-composition with the opposite $\tau^\op\colon\thinspace \cat{S}^\op\hookrightarrow \cat{T}^\op$ of~$\tau$:
\[ \cat{Pe}(\tau,\cat{D})\colon\thinspace \Pe\to \cat{Pe}(\cat{S},\cat{D}). \]
The same holds when replacing the category~\(\cat{T}\) by~\(\cat{S}\) and \(\D\) by~\(\Omega_n\), for \(n\geq 0\).

For each $n\geq0$, we consider the following pullback square of categories and functors: 

\begin{equation*}
	\begin{tikzcd}
		\Pe\times_{\cat{Pe}(\cat{S},\cat{D})}\cat{Pe}(\cat{S},\Omega_n) \arrow[r] \arrow[d, "\pi_n", swap]
		\arrow[dr, phantom, "\usebox\pullbacksquare", very near start, color=black]
		& \cat{Pe}(\cat{S},\Omega_n) \arrow[d, "{\cat{Pe}(\cat{S},\iota_n)}"]
		\\ \Pe \arrow[r, "{\cat{Pe}(\tau,\cat{D})}", swap]
		& \cat{Pe}(\cat{S},\cat{D}).
	\end{tikzcd}
\end{equation*}
We define~$\widehat{\Omega}_n$ to be the image of the pullback category under $\pi_n$, so this is the subcategory
\[ \widehat{\Omega}_n\coloneqq \pi_n(\Pe\times_{\cat{Pe}(\cat{S},\cat{D})}\cat{Pe}(\cat{S},\Omega_n))\subseteq\Pe. \]

\begin{lemma}
	For every integer $n\geq0$, $\widehat{\Omega}_n$~is the full subcategory of $\Pe$ whose objects are those persistent narratives $F\colon\thinspace\cat{T}^\op\to\D$ whose value $F([a,b])$ at any time interval $[a,b]$ in~$\cat{S}$ is contained in~$\Omega_n$.	
\end{lemma}

\begin{theorem}\label{thm:termporalized-sd-category}
	Let again $\cat{T}$ be a finite discrete time category, let $\tau\colon\thinspace\cat{S}\hookrightarrow\cat{T}$ be the inclusion of a sub-join-semilattice  $\cat{S}\subseteq\cat{T}$, and let $(\D,\cG,\Omega)$ be a spined sd-category satisfying~(T\ref{axiom:T1}). Furthermore, suppose that the following conditions hold.
	\begin{enumerate}[leftmargin=0.5in]
		\renewcommand{\labelenumi}{(T\theenumi)}
		\setcounter{enumi}{1}
		\item\label{axiom:T2} The category~$\D$ is cocomplete, that is, it admits all small colimits.
		
		\item\label{axiom:T3} The inclusion functor \(\Pe\hookrightarrow[\cat{T}^\op,\D]\) admits a left adjoint \(\mathfrak{S}\colon[\cat{T}^\op,\D]\to\Pe\) whose restriction to \(\Pe\) is the identity functor.
	\end{enumerate}
	Then the pair $(\Pe,\cG)$ is an sd-category.
\end{theorem}

Since \(\mathfrak{S}\) is both a left adjoint and a retraction of the inclusion functor, it can be regarded as a sheafification functor. In Example~\ref{eg:ordinary-tree-width}, where \(\D\) is the category~\(\Grph\) of graphs, \(\mathfrak{S}\) is induced by the usual sheafification functor for set-valued presheaves.

By adding one final axiom, we arrive at the notion of a temporalized spined sd-category.

\begin{theorem}\label{thm:temporalization-of-spined-sd-category}
	As before, let $\cat{T}$ be a finite discrete time category, and let $\tau\colon\thinspace\cat{S}\hookrightarrow\cat{T}$ be the inclusion of a sub-join-semilattice $\cat{S}\subseteq\cat{T}$.
	Let $(\D,\cG,\Omega)$ be a spined sd-category that satisfies the axioms (T\ref{axiom:T1}), (T\ref{axiom:T2}) and (T\ref{axiom:T3}) as well as 
	\begin{enumerate}[leftmargin=0.5in]
		\renewcommand{\labelenumi}{(T\theenumi)}
		\setcounter{enumi}{3}
		\item\label{axiom:T4} The category~$\D$ admits pushout squares along monomorphisms and these are also pullback squares. Moreover, monomorphisms are stable under pushouts.
	\end{enumerate}
	Then the subcategories $\widehat{\Omega}_n\subseteq\Pe$ with $n\geq0$ define a spine $\widehat{\Omega}$ on $\Pe$, thereby exhibiting the sd-category $(\Pe,\cG)$ of Theorem~\ref{thm:termporalized-sd-category} as a spined sd-category $(\Pe,\cG,\widehat{\Omega})$.
\end{theorem}

Theorem~\ref{thm:temporalization-of-spined-sd-category} allows measures of complexity to be lifted from the static to the time-varying setting. A complete proof is given in~\cite{decomposingtimevaryingdata2026}.

\subsubsection{Temporalizing tree-width: examples}
\label{subsubsubsection:temporalizing-tree-width-examples}

In ordinary graph theory, the notion of tree-width is typically defined by means of tree-decompositions, as explained in Paragraph~\ref{paragraph:tree-width}. The category-theoretical framework of spined sd-categories~\cite{Bumpus-et-al_Structured-Decompositions} extends this concept to arbitrary structured decompositions. More precisely, given a spined sd-category \(\Gamma=(\D,\cG,\Omega)\), each object of~\(\D\) is assigned a \textit{\(\Gamma\)-size}~\cite[Definition~2.5.4]{Bumpus-et-al_Structured-Decompositions}. We restrict attention to spined sd-categories in which the subcategories \(\Omega_n\subseteq\D\) are full, in which case the definition simplifies as follows.

\begin{definition}
	The \textit{$\Gamma$-size} of an object $X\in\Ob(\D)$ is the minimum non-negative integer~$n$ for which there exists an object $W\in\Ob(\Omega_n)$ together with a monomorphism $X\rightarrowtail W$ in~$\D$. The resulting map $ s_\Gamma\colon\thinspace \Ob(\D)\to\N_0$, is called the \textit{size function} of~$\Gamma$.
\end{definition}

The size function allows the notions of tree-width for tree-decompositions and graphs to be extended to width notions for structured decompositions and objects of arbitrary spined sd-categories, respectively.

\begin{definition}
	The \textit{width}~$w_\Gamma(d)$ of a structured decomposition $d\colon\thinspace \int J\to\D$ with $J\in\cG$ is the maximum $\Gamma$-size of its bags $d(v)$ minus one:
	\[ w_\Gamma(d)\coloneqq \max_{v\in V(J)} s_\Gamma(d(v)) -1. \]
	Now, the \textit{$\Gamma$-width}~$w_\Gamma(X)$ of an object $X$ of~$\D$ is defined as the minimum $\Gamma$-widths of all structured decompositions $d\colon\thinspace \int J\to\D$ with $J\in\cG$ whose colimit is isomorphic to~$X$:
	\[ w_\Gamma(X)\coloneqq \min_{\substack{d\colon\!\! \int\! J\to\D \textup{ with} \\ J\in\cG,\, \colim d \,\cong\, X}} w_\Gamma(d). \]
\end{definition}

Section~3 of~\cite{Bumpus-et-al_Structured-Decompositions} shows that many classical graph width parameters are recovered as instances of this general notion by choosing appropriate spined sd-categories. This justifies spined sd-categories as a unified framework for generalized tree-widths.

To justify the temporalization of spined sd-categories and the resulting notions of time-varying structured decompositions and \(\Gamma\)-width, we consider three classical graph width parameters. For each, we exhibit a spined sd-category whose temporalization via Theorem~\ref{thm:temporalization-of-spined-sd-category} yields a natural time-varying analogue of the corresponding width notion. Specifically, we consider ordinary tree-width (Definition~\ref{defi:width-tree-width}), complemented tree-width, and the tree independence number.

\begin{example}\label{eg:ordinary-tree-width} (Ordinary tree-width).
	Proposition~3.1.1 in~\cite{Bumpus-et-al_Structured-Decompositions} shows that the category $\Grph$ together with the class of all trees and the full subcategories $\Omega_n\subseteq\Grph$ containing the complete graphs of order at least~$n$ is a spined sd-category whose associated notion of width is the ordinary tree-width of a graph. To apply our temporalization method of Theorem~\ref{thm:temporalization-of-spined-sd-category}, we need to replace $\Grph$ by the category of \textit{reflexive} graphs, meaning those graphs with precisely one loop at each vertex.
	
	\begin{proposition}
		For any finite discrete time category~$\cat{T}$ and any sub-join-semilattice $\cat{S}\subseteq\cat{T}$, $\Gamma$~induces a temporalized spined sd-category $\widehat{\Gamma}\coloneqq ( \cat{Pe}(\cat{T},\Grefl),\{\textup{trees}\},\widehat{\Omega} )$, whose size function takes any persistent narrative $X\colon\thinspace \cat{T}^\op\to\Grefl$ to the maximum order among the graphs~$X(s)$ with $s\in\Ob(\cat{S})$,
		\[ s_{\widehat{\Gamma}}(X) = \max_{s\in\Ob(\cat{S})} \abs{X(s)}. \]
	\end{proposition}
	
	The $\widehat{\Gamma}$-size of a structured decomposition $d\colon\thinspace\int J\to \cat{Pe}(\cat{T},\Grefl)$ for any tree~$J$ can be shown to be
	\[ w_{\widehat{\Gamma}}(d) = \max_{s\in\Ob(\cat{S})} w_{\Gamma}(d_s) \]
	where $d_s\colon\thinspace \int J\to\Grefl$ is obtained from $d$ by evaluation at $s\in\Ob(\cat{S})$. Since $w_{\Gamma}$ is the ordinary tree-width of a graph, we thus recover maximum tree-width, which is a well-known generalization of tree-width for temporal graphs. 
\end{example}

\begin{example} (Complemented tree-width).
	The \textit{complemented tree-width} of a graph~$G$ is the ordinary tree-width of its \textit{complement}~$\overline{G}$, the graph with the same vertex set as~$G$ and with an edge between two vertices if and only if these vertices are non-adjacent in~$G$. Let $\overline{\Grph}$ be the category whose objects are the undirected graphs and whose morphisms from $G$ to~$H$ are graph morphisms of complements $\overline{G}\to\overline{H}$. By \cite[Proposition 3.3.4]{Bumpus-et-al_Structured-Decompositions}, $\Gamma\coloneqq (\overline{\Grph}, \{\textup{trees}\},\Omega)$ is a spined sd-category where $\Omega_n$ is the full subcategory of $\overline{\Grph}$ whose objects are all edgeless graphs of order at most~$n$. Its associated width notion is complemented tree-width.
	
	\begin{proposition}
		For any sub-join-semilattice~$\cat{S}$ of a finite discrete time category~$\cat{T}$, $\Gamma$~induces a time-varying spined sd-category $( \cat{Pe}(\cat{T},\overline{\Grph}),\{\textup{trees}\},\widehat{\Omega} )$. Its width notion is given by the maximum complemented tree-width over~$\cat{S}$.
	\end{proposition}	
\end{example}

Modifying the spine~$\Omega$ on~$\overline{\Grph}$ appropriately will give rise to the tree independence number of a graph, discussed in the next example.

\begin{example} (Tree independence number).
	The \textit{independence number} of a graph is the maximum number of pairwise non-adjacent vertices. The \textit{tree independence number} of a tree decomposition is the maximum independence number of its bags minus one, and the \textit{tree independence number} of a graph is the minimum one among all its tree decompositions. 
	For any $n\geq0$, let $\Omega_n^\alpha$ denote the full subcategory of $\overline{\Grph}$ whose objects are the graphs of independence number $\alpha(G)\leq n$. Then by \cite[Proposition 3.4.5]{Bumpus-et-al_Structured-Decompositions}, they define a spine $\Omega^\alpha$ on $\overline{\Grph}$ making the triple $(\overline{\Grph},\{\textup{trees}\},\Omega^\alpha)$ to a spined sd-category, whose width notion is the tree independence number.
	
	\begin{proposition}
		Given $\cat{T}$ and $\cat{S}$ as usual, the above spined sd-category induces a temporalization $( \cat{Pe}(\cat{T}, \overline{\Grph}),\{\textup{trees}\},\widehat{\Omega}^\alpha )$. Its size function is given by the maximum independence number of graphs, so its associated width notion recovers the maximum tree independence number over~$\cat{S}$.
	\end{proposition}
\end{example}

\subsection{Vignette 3: Temporal Cellular Sheaves: Modelling Multi-agent Systems with Switching Topologies}

This vignette outlines an ongoing research program developing category-theoretic and sheaf-theoretic methods for control, with current applications to control barrier functions and multi-agent systems~\cite{CurrierLealEtAl2026,temporal-cellular-sheaves2026}.  Building on these ideas, we focus here on multi-agent systems with switching communication topologies~\cite{temporal-cellular-sheaves2026}. Such systems arise in applications ranging from robotic swarms and autonomous vehicles to sensor networks and distributed optimization, where the communication graph changes over time due to mobility, communication failures, or environmental constraints~\cite{MesbahiEgerstedt2010,OlfatiSaber2007,Bullo2009,Moreau2005}.

A common approach represents the communication topology by a graph whose vertices correspond to agents and whose edges represent communication links~\cite{MesbahiEgerstedt2010,Bullo2009}. This graph captures the connectivity of the network and provides the underlying structure for many distributed control algorithms. The communication graph, however, captures only who communicates with whom. The dynamics of the agents, the information they exchange, the sensing mechanisms, and the transformations performed along communication links must be modeled separately. Because these additional structures are usually introduced in an application-specific manner, it is difficult to develop a unified mathematical framework encompassing heterogeneous multi-agent systems.

\subsubsection{Cellular sheaves for multi-agent systems}

Cellular sheaves offer a mathematical model for the additional structures required in heterogeneous multi-agent systems~\cite{hanks2025heterogeneousmultiagent,distributed-mas-coordination-2026}. They assign (potentially different) vector spaces to agents and communication links, together with linear maps describing how information is measured, communicated, or transformed locally. This construction naturally accommodates heterogeneous state spaces, sensing mechanisms, communication protocols, and local information-processing rules. In this way, the communication graph specifies \emph{who communicates with whom}, while the cellular sheaf specifies \emph{what information is associated with each agent and communication link, and how that information is transformed}.

Categorically, a cellular sheaf is simply a functor
\[
\mathcal G:\incG\longrightarrow\Vect,
\]
where \(\incG\) denotes the incidence category of the communication (directed) graph \(G\). Its objects are the faces of \(G\) (vertices and edges), and its morphisms are precisely the incidence relations: for every edge \(e=(u,v)\), there are morphisms \(u\xrightarrow{t_e} e\) and \(v\xrightarrow{h_e} e\). Figure~\ref{fig:cellular-sheaf-components} illustrates this construction.

\begin{figure}[htbp]
\centering
\includegraphics[width=\columnwidth]{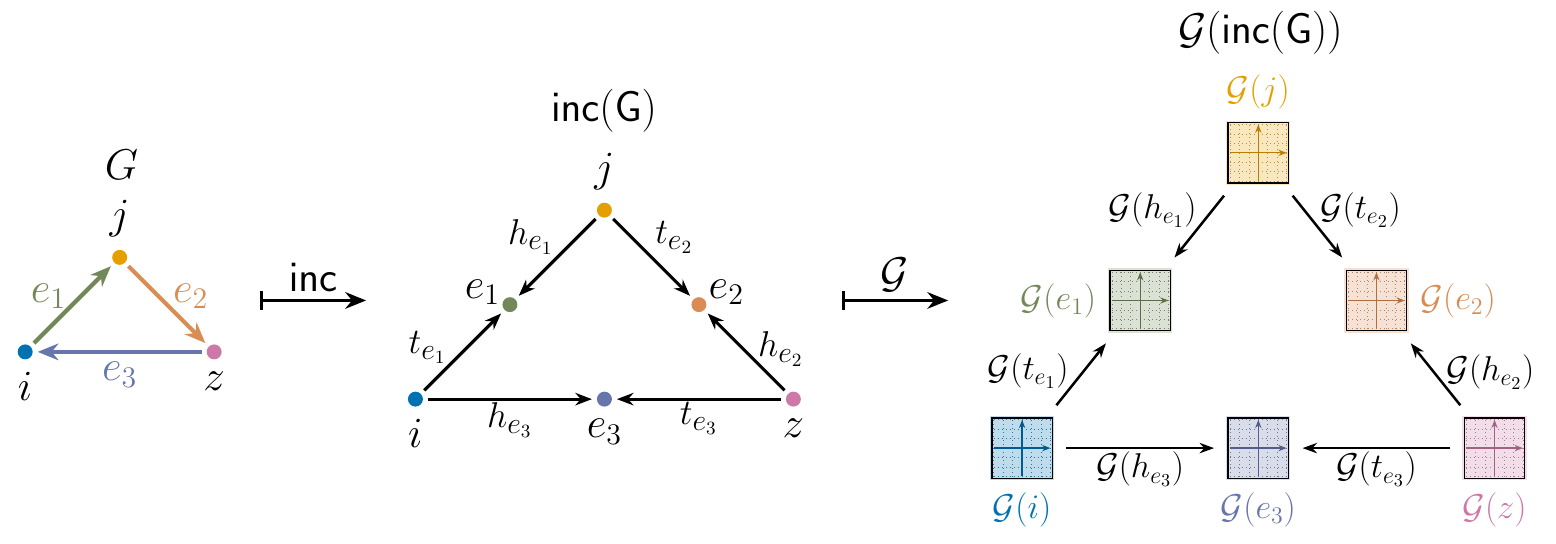}
\caption{From left to right: the directed communication graph \(G\), whose vertices represent agents and whose edges \(e_1, e_2\) and \(e_3\) represent directed communication links; its incidence category \(\incG\), whose objects are the faces of \(G\), namely its vertices and edges, and whose morphisms record the incidence relations; and the diagram determined by the cellular sheaf \(\mathcal G:\incG\to\Vect\). In this example, \(\mathcal G\) assigns the vector space \(\mathbb R^2\) to every object of \(\incG\), equivalently to every face of \(G\), together with linear maps associated with the incidence morphisms. These maps encode the local sensing, communication, and information-processing relations of the network.}

\label{fig:cellular-sheaf-components}
\end{figure}

\begin{remark}[Cellular sheaves as structured co-decompositions]
Notice that cellular sheaves are instances of the structured co-decompositions considered in Section~\ref{subsec:decompositions} since the incidence category \(\incG\) of a graph $G$ is  isomorphic to $(\int G)^{op}$. Thus, a cellular sheaf \(\mathcal G\colon\incG\to\Vect\) is equivalently a contravariant functor from $\hat{\mathcal G} \colon (\int G)^{op} \to \Vect$, and hence a $G$-shaped structured co-decomposition of vector spaces. Equivalently, after taking opposites, it may be regarded as a functor \(\mathcal G^{op}\colon \int G \to \Vect^{op}\). When the cellular sheaf maps are epimorphisms in $\Vect$, they become monomorphisms in $\Vect^{op}$, so that $\mathcal G^{op}$ is a structured decomposition in the sense used earlier in this chapter. To match the terminology of control theory, multi-agent systems, and the cellular-sheaf literature, we refer to these objects as cellular sheaves throughout this section.
\end{remark}

The functorial viewpoint is essential because it guarantees that the local models assigned to agents and communication links are compatible with the communication topology. Rather than specifying local state spaces, sensing maps, and communication rules independently, functoriality requires them to respect the incidence relations of the graph. Consequently, the resulting cellular sheaf represents a globally consistent information-processing architecture for the multi-agent system.

\subsubsection{Category of cellular sheaves}

To model systems whose communication topology changes over time, we must also specify how one cellular sheaf transforms into another. This requires a suitable notion of morphism between cellular sheaves, capturing both the evolution of the underlying communication graph and the induced transformation of the associated local state data. This leads naturally to the category of cellular sheaves.

\begin{definition}[Category of cellular sheaves]\label{def:cellsh}
The category \(\CellSh\) is defined as follows.
\begin{itemize}
    \item Its objects are cellular sheaves \(\mathcal G:\incG\longrightarrow\Vect,\) where \(G\) is a finite directed graph.

    \item A morphism \((\mathcal T,\alpha):\mathcal G\to\mathcal H\) consists of a functor \(\mathcal T:\incG\to\incH\) and a natural transformation \(\alpha:\mathcal G\Rightarrow\mathcal H\circ\mathcal T\), as shown below.
% https://q.uiver.app/#q=WzAsMyxbMCwwLCJcXGluY0ciXSxbMiwwLCJcXGluY0giXSxbMSwyLCJcXFZlY3QiXSxbMCwxLCJ7XFxtYXRoY2FsIFR9Il0sWzAsMiwie1xcbWF0aGNhbCBHfSIsMl0sWzEsMiwie1xcbWF0aGNhbCBIfSJdLFs0LDUsIlxcYWxwaGEiLDAseyJvZmZzZXQiOi0yfV1d
\[\begin{tikzcd}
	\incG && \incH \\
	\\
	& \Vect
	\arrow["{{\mathcal T}}", from=1-1, to=1-3]
	\arrow[""{name=0, anchor=center, inner sep=0}, "{{\mathcal G}}"', from=1-1, to=3-2]
	\arrow[""{name=1, anchor=center, inner sep=0}, "{{\mathcal H}}", from=1-3, to=3-2]
	\arrow["\alpha", shift left=2, Rightarrow, from=0, to=1]
\end{tikzcd}\]
\end{itemize}
\end{definition}
We first examine the functor \(\mathcal T:\incG\to\incH\) appearing in a morphism \((\mathcal T,\alpha):\mathcal G\to\mathcal H\). Recall that the objects of \(\incG\) are the agents and communication links of \(G\), while its nonidentity morphisms encode incidences between them. Thus, \(\mathcal T\) specifies how the communication topology \(G\) is represented inside, or transformed into, the communication topology \(H\). Because \(\mathcal T\) is functorial, it preserves incidence relations: whenever an agent is incident to a communication link in \(G\), its image must be incident to the image of that link in \(H\).

The following examples illustrate increasingly sophisticated ways in which a functor between incidence categories may arise in applications, beginning with subsystem inclusions, progressing to a coarse-graining of agents into teams, and culminating in switching communication topologies.

\begin{example}[A subsystem of a system]
Let \(G\) be a subgraph of \(H\). The inclusions of vertices and edges determine a functor \(\mathcal I:\incG\hookrightarrow\incH\) that sends every agent and communication link of \(G\) to the corresponding cell of \(H\). Functoriality follows because incidences in \(G\) remain incidences in \(H\).
\end{example}

\begin{example}[Aggregation of agents into teams]
Hierarchical control often requires reasoning simultaneously at different levels of abstraction~\cite{SCATTOLINI2009723,BaiGeorgeChakrabortty2020,9482931}. At the lower level, individual agents communicate through a detailed network, while at the higher level, groups of agents are treated as teams interacting through a coarser communication topology. A functor between the corresponding incidence categories relates these two descriptions.

Consider the path graph \(G=P^4=1\!-\!2\!-\!3\!-\!4\), representing four communicating agents. Suppose that agents \(1\) and \(2\) form team \(A\), while agents \(3\) and \(4\) form team \(B\). The team-level communication graph \(H\) therefore consists of two vertices, \(A\) and \(B\), connected by a single edge \(AB\). The aggregation is described by a functor
\[
\mathcal T:\incG\longrightarrow\incH
\]
defined by
\[
\mathcal T(1)=\mathcal T(2)=A,\qquad
\mathcal T(3)=\mathcal T(4)=B,
\]
and
\[
\mathcal T(12)=A,\qquad
\mathcal T(23)=AB,\qquad
\mathcal T(34)=B.
\]

\begingroup
% Color-blind-friendly colors from the Okabe--Ito palette
\definecolor{teamA}{RGB}{0,114,178}       % Blue
\definecolor{teamB}{RGB}{213,94,0}        % Vermilion
\definecolor{interteam}{RGB}{0,158,115}   % Bluish green

\begin{tikzcd}[
    column sep=1.8em,
    row sep=2.2em,
    every arrow/.append style={line width=0.6pt},
    execute at end picture={
        \node at
        ($(functorleft.center)!0.5!(functorright.center)-(0,3mm)$)
        {\Large$\xrightarrow{\;\;\mathcal T\;\;}$};
    }
]
&
{\color{teamA}12}
&&
{\color{interteam}23}
&&
{\color{teamB}34}
&&
|[alias=functorleft]| {}
&
|[alias=functorright]| {}
&&
{\color{interteam}AB}
&
\\
{\color{teamA}1}
&&
{\color{teamA}2}
&&
{\color{teamB}3}
&&
{\color{teamB}4}
&&&
{\color{teamA}A}
&&
{\color{teamB}B}
%
% Incidence morphisms of inc(P^4)
\arrow[teamA, from=1-2, to=2-1]
\arrow[teamA, from=1-2, to=2-3]
\arrow[interteam, from=1-4, to=2-3]
\arrow[interteam, from=1-4, to=2-5]
\arrow[teamB, from=1-6, to=2-5]
\arrow[teamB, from=1-6, to=2-7]
%
% Identity morphisms of inc(P^4)
\arrow[teamA, loop above, from=1-2, to=1-2]
\arrow[interteam, loop above, from=1-4, to=1-4]
\arrow[teamB, loop above, from=1-6, to=1-6]
\arrow[teamA, loop below, from=2-1, to=2-1]
\arrow[teamA, loop below, from=2-3, to=2-3]
\arrow[teamB, loop below, from=2-5, to=2-5]
\arrow[teamB, loop below, from=2-7, to=2-7]
%
% Incidence morphisms of inc(A-B)
\arrow[interteam, from=1-11, to=2-10]
\arrow[interteam, from=1-11, to=2-12]
%
% Identity morphisms of inc(A-B)
\arrow[interteam, loop above, from=1-11, to=1-11]
\arrow[teamA, loop below, from=2-10, to=2-10]
\arrow[teamB, loop below, from=2-12, to=2-12]
\end{tikzcd}
\endgroup

The color coding displays the aggregation. The objects \(1\), \(2\), and \(12\) share the color of team \(A\), while \(3\), \(4\), and \(34\) share the color of team \(B\). Accordingly, the intra-team links \(12\) and \(34\) are mapped to the team objects \(A\) and \(B\), respectively. Their incidence morphisms are therefore mapped to the corresponding identity morphisms, shown as loops in the team-level category. This expresses that communication within each team is black-boxed in the coarse description.

By contrast, the link \(23\) and its incidence morphisms share the color of the edge \(AB\) and its incidence morphisms. Thus, the communication between the two teams remains visible as a nontrivial edge at the team level.
\end{example}

\begin{example}[Switching topologies]\label{ex:switching-topologies}
Consider a multi-agent system whose communication topology changes over time. Let \(G_{t_i}\) and \(G_{t_k}\) denote the communication graphs at two time instants \(t_i<t_k\). Both may be viewed as communication subgraphs of a larger graph \(G\) describing every communication link that appears during the entire time horizon. The corresponding inclusion functors
\[
\inc{G_{t_i}}
\hookrightarrow
\incG
\hookleftarrow
\inc{G_{t_k}}.
\]
admit a pullback
\[
\inc{G_{t_i}}
\hookleftarrow
\inc{G_{t_i}}\times_{\incG}\inc{G_{t_k}}
\hookrightarrow
\inc{G_{t_k}},
\]
whose apex identifies precisely the agents and communication links common to both topologies. The resulting span therefore represents the communication subsystem that persists across the transition from \(G_{t_i}\) to \(G_{t_k}\).
\end{example}

The second component of a morphism of cellular sheaves is given by a natural transformation \(\alpha:\mathcal G\Rightarrow\mathcal H\circ\mathcal T\). For every incidence morphism \(m:x\to y\) in \(\incG\), naturality requires the commutativity of
\[
\begin{tikzcd}
\mathcal G(x) \arrow[r,"\mathcal G(m)"] \arrow[d,"\alpha_x"']
& \mathcal G(y) \arrow[d,"\alpha_y"]\\
\mathcal H(\mathcal T(x)) \arrow[r,"\mathcal H(\mathcal T(m))"']
& \mathcal H(\mathcal T(y)).
\end{tikzcd}
\]

Thus, performing the local computation prescribed by \(\mathcal G\) and then communicating the resulting information via \(\alpha\) to \(\mathcal H\) yields the same result as first communicating the local data from \(\mathcal G\) to \(\mathcal H\) and then perform the corresponding local computation in \(\mathcal H\). In this sense, a morphism \((\mathcal T,\alpha):\mathcal G\to\mathcal H\) makes the local information-processing rules of the two systems compatible with the information processing protocol within each system.

\begin{example}[Persistent information across a topology change]
Consider the switching-topology setting of Example~\ref{ex:switching-topologies}. We use this situation to illustrate the form of morphisms in the category of cellular sheaves. Let
% https://q.uiver.app/#q=WzAsNSxbMCwxLCJcXGluY3tHX3t0X2l9Xnt0X2l9fSJdLFs0LDEsIlxcaW5je0dfe3Rfa31ee3Rfa319Il0sWzIsMCwiXFxpbmN7R197dF9pfV57dF9rfX0iXSxbNCwyXSxbMiwzLCJcXFZlY3QiXSxbMiwwLCJcXG1hdGhjYWwgSV9pIiwyXSxbMCw0LCJcXG1hdGhjYWwgR197dF9pfV57dF9pfSIsMl0sWzIsMSwiXFxtYXRoY2FsIElfayIsMCx7ImxhYmVsX3Bvc2l0aW9uIjo3MH1dLFsxLDQsIlxcbWF0aGNhbCBHX3t0X2t9Xnt0X2t9Il0sWzIsNCwiXFxtYXRoY2FsIEdfe3RfaX1ee3Rfa30iLDIseyJsYWJlbF9wb3NpdGlvbiI6MzB9XSxbOSw2LCJcXGFscGhhX2kiLDAseyJzaG9ydGVuIjp7InNvdXJjZSI6MjAsInRhcmdldCI6MjB9fV0sWzksOCwiXFxhbHBoYV9rIiwyLHsic2hvcnRlbiI6eyJzb3VyY2UiOjIwLCJ0YXJnZXQiOjIwfX1dXQ==
\[\begin{tikzcd}[column sep=small]
	&& {\inc{G_{t_i}^{t_k}}} && \\
	{\inc{G_{t_i}^{t_i}}} &&&& {\inc{G_{t_k}^{t_k}}} \\
	&&&& {} \\
	&& \Vect
	\arrow["{\mathcal I_i}"', from=1-3, to=2-1]
	\arrow["{\mathcal I_k}"{pos=0.7}, from=1-3, to=2-5]
	\arrow[""{name=0, anchor=center, inner sep=0}, "{\mathcal G_{t_i}^{t_k}}"'{pos=0.3}, from=1-3, to=4-3]
	\arrow[""{name=1, anchor=center, inner sep=0}, "{\mathcal G_{t_i}^{t_i}}"', from=2-1, to=4-3]
	\arrow[""{name=2, anchor=center, inner sep=0}, "{\mathcal G_{t_k}^{t_k}}", from=2-5, to=4-3]
	\arrow["{\alpha_i}", between={0.2}{0.8}, Rightarrow, from=0, to=1]
	\arrow["{\alpha_k}"', between={0.2}{0.8}, Rightarrow, from=0, to=2]
\end{tikzcd}\]
where \(\mathcal G_{t_i}^{t_i}\), \(\mathcal G_{t_i}^{t_k}\), and \(\mathcal G_{t_k}^{t_k}\) are cellular sheaves on the communication topologies \(\inc{G_{t_i}}\), \(\inc{G_{t_i}^{t_k}}\), and \(\inc{G_{t_k}}\), respectively. Since the communication topologies and the functors \(\mathcal I_i\) and \(\mathcal I_k\) have already been specified, defining morphisms of cellular sheaves amounts to specifying natural transformations
\[
\alpha_i:
\mathcal G_{t_i}^{t_k}
\Rightarrow
\mathcal G_{t_i}^{t_i}\circ\mathcal I_i,
\qquad
\alpha_k:
\mathcal G_{t_i}^{t_k}
\Rightarrow
\mathcal G_{t_k}^{t_k}\circ\mathcal I_k,
\]
together with these functors, define the morphisms of cellular sheaves
\[
\mathcal G_{t_i}^{t_i}
\xleftarrow{(\mathcal I_i,\alpha_i)}
\mathcal G_{t_i}^{t_k}
\xrightarrow{(\mathcal I_k,\alpha_k)}
\mathcal G_{t_k}^{t_k}.
\]

The natural transformations \(\alpha_i\) and \(\alpha_k\) assign compatible linear maps to every persistent agent and communication link. Consequently, the communication subsystem that survives the topology change is accompanied by a coherent identification of the local state spaces, sensing maps, and communication constraints at both time instants.
\end{example}

\subsubsection{Temporal cellular sheaves and switching topologies}

To model multi-agent systems with switching communication topologies as temporal narratives, we need the category of cellular sheaves to have the required structure. The following lemma, proved in~\cite{temporal-cellular-sheaves2026}, establishes precisely this fact.

\begin{lemma}\label{lem:cellsh-limits}
The category \(\CellSh\) admits pullbacks and pushouts.
\end{lemma}

Consequently, all the constructions developed in Section~\ref{subsec:fixed-points} apply with
\(\D=\CellSh\).

\begin{definition}[Temporal cellular sheaf]
A \emph{temporal cellular sheaf} is a \(\CellSh\)-valued narrative, that is, a sheaf
\[
\mathcal F:\cat T^{op}\longrightarrow\cotw{\CellSh}.
\]
\end{definition}

Thus, each interval \([a,b]\in\cat T\) is assigned a morphism of cellular sheaves \(P_a^b\xrightarrow{\gamma_a^b}C_a^b,\) where \(P_a^b\) and \(C_a^b\) model the persistent and cumulative aspects of the multi-agent system over the interval \([a,b]\), respectively, while \(\gamma_a^b\) relates the two descriptions. The sheaf condition reconstructs the persistent component by pullbacks and the cumulative component by pushouts, so that both the communication topology and the associated sensing and interaction maps evolve coherently through time.

The principal advantage of temporal cellular sheaves is that they model not only the communication topology at each time instant but also the structural relationships between communication topologies across time intervals. Thus, instead of viewing a switching multi-agent system as a sequence of independent communication graphs, the narrative records how information persists and accumulates as the communication architecture evolves.

Within this framework, distributed coordination problems such as consensus, formation control, and target tracking can be formulated categorically. The dynamics of the agents evolve on the vector spaces assigned by the cellular sheaves, while the temporal narrative specifies how these local dynamical models are related as the communication topology changes.

Current work investigates the stability theory of temporal cellular sheaves, building on recent sheaf-theoretic approaches to distributed control together with classical Lyapunov methods for switching systems. The objective is to characterize how changes in the communication topology interact with the evolution of distributed state variables, and to establish sufficient conditions under which target-tracking errors remain bounded and converge despite topology switches. This research program is currently being developed in~\cite{temporal-cellular-sheaves2026}.

\section{Conclusion and future directions}
\label{section:future}

One of the strengths of the narrative framework is that it is largely independent of the nature of the objects evolving through time.  Once an appropriate target category has been identified, the same theory immediately yields a coherent framework for modeling temporal phenomena across a wide range of applications.  The three research directions presented in this chapter illustrate this flexibility from complementary perspectives: the first investigates the categorical structure of the persistence--accumulation adjunction through the category of narratives and its induced factorization of the adjunction, the second explores temporal analogues of structured decompositions, and the third instantiates the framework in the category of cellular sheaves to model multi-agent systems with switching communication topologies. 

\paragraph{Towards a characterization of the fixed points of the adjunction}
The factorization of the persistence--accumulation adjunction through the category of narratives introduces a setting in which persistence and accumulation are no longer viewed as separate constructions, but as two complementary descriptions of the same temporal object. The resulting rigidity classification measures the extent to which these descriptions determine one another. Rigid narratives coincide with the canonical completions induced by the adjunction in both directions. Left-rigid narratives preserve their persistent description under the persistence--accumulation round trip, right-rigid narratives preserve their cumulative description, while loose narratives lose information in both directions. 

The rigidity classification also opens several directions for future research. A first natural question is to determine which rigidity classes arise in a given ambient category and how they reflect its structural properties. More generally, one may seek intrinsic categorical characterizations of left, right, and rigid narratives, as well as investigate how rigidity behaves under products, limits, colimits, functorial changes of the ambient category, and other categorical constructions. Beyond their intrinsic categorical interest, these questions contribute to a broader understanding of information-preserving changes of perspective, including those arising in data science.

\paragraph{Structured decompositions of cumulative narratives}

By virtue of the tight relation between persistent and cumulative narratives, as revealed in Section~\ref{subsection:cumulative-and-persistent-perspectives}, it is desirable to temporalize structured decompositions and width also from the cumulative point of view. 
More specifically, under dual assumptions of Theorem~\ref{thm:termporalized-sd-category}, we may form a pullback square
	\begin{equation*}
	\begin{tikzcd}
		\Cu \times_{\cat{Cu}(\cat{S},\cat{D})}\cat{Cu}(\cat{S},\Omega_n) \arrow[r] \arrow[d, "\pi_n", swap]
		\arrow[dr, phantom, "\usebox\pullbacksquare", very near start, color=black]
		& \cat{Cu}(\cat{S},\Omega_n) \arrow[d, "{\cat{Cu}(\cat{S},\iota_n)}"]
		\\ \Cu \arrow[r, "{\cat{Cu}(\tau,\cat{D})}", swap]
		& \cat{Cu}(\cat{S},\cat{D}).
	\end{tikzcd}
	\end{equation*}
	We define $\check{\Omega}_n\subseteq \Cu$ to be the image category of the functor~$\pi_n$. If we dualize the assumptions of Theorem~\ref{thm:temporalization-of-spined-sd-category}, do we again obtain a spined sd-category of the form $(\Cu,\cG,\check{\Omega})$? How does the proof change?

A second direction is to investigate the relationship between the persistent and cumulative temporalizations of spined sd-categories.
As exposed in Theorem~\ref{thm:adjunction}, given any time category~$\cat{T}$ and a category~$\D$ that is both complete and cocomplete, then there exists a pair of adjoint functors 
	\begin{equation*}
	\begin{tikzcd}
		\Pe
		\arrow[bend left=30]{r}
		\arrow[r, phantom, "\raisebox{-0.2ex}{\rotatebox{270}{$\dashv$}}"]
		& \Cu.
		\arrow[bend left=30]{l}
	\end{tikzcd}
	\end{equation*} 
It would be interesting to investigate the following questions, assuming the setting of Theorem~\ref{thm:temporalization-of-spined-sd-category} together with the dual assumptions:
	
\begin{itemize}
    \item 	Are the above adjoint functors sd-functors in the sense of \cite[Definition 2.7.1]{Bumpus-et-al_Structured-Decompositions}? 
	
	\item Are they even width-preserving in the sense of \cite[Definition 2.7.8]{Bumpus-et-al_Structured-Decompositions}? 
	
	\item If not, can this be guaranteed by modifying the assumptions appropriately?
\end{itemize}
	
	It would be desirable to have an adjoint pair of width-preserving sd-functors between the persistent and cumulative temporalized spined sd-categories. This would not only reveal a fundamental relation between the persistent and cumulative perspectives but also allow a convenient transfer between concepts and results for persistent narratives to cumulative ones and vice versa.

\paragraph{Temporal cellular sheaves: modelling multi-agent systems with switching topology}
Temporal cellular sheaves demonstrate how the narrative framework applies naturally to heterogeneous multi-agent systems with switching communication topologies. Choosing cellular sheaves as the target category allows the framework to encode not only the evolution of the communication architecture but also the heterogeneous sensing, communication, and information-processing structures associated with it. This establishes a categorical foundation for studying distributed control problems on time-varying communication networks.

A natural next step is to endow temporal cellular sheaves with dynamical systems evolving on the vector spaces assigned to their cells. This raises the problem of understanding how the dynamics interact with topology changes and with the persistent and cumulative structures encoded by the narrative. Of particular interest is the development of a stability theory for the sheaf-theoretic description of switching communication networks, leading to conditions under which distributed coordination objectives---such as consensus, formation maintenance, or target tracking---remain stable despite changes in the communication topology. This research program is currently under development in~\cite{temporal-cellular-sheaves2026}.

\backmatter

\section*{Declarations}

\paragraph{Funding}
Benjamin Merlin Bumpus was supported by the São Paulo Research Foundation (FAPESP), grant 2025/16921-5.

\paragraph{Conflict of interest}
The authors declare that they have no competing interests.

\paragraph{Data availability}
No datasets were generated or analyzed during the current study.

\paragraph{Materials availability}
Not applicable.

\paragraph{Code availability}
Not applicable.

\paragraph{Author contributions}
The authors' contributions are described according to the CRediT (Contributor Roles Taxonomy) as follows. Conceptualization: W.L., B.B. Methodology: W.L., B.B. Formal analysis: W.L., B.B., J.N., J.G. Investigation: W.L., B.B., J.N., J.G. Resources: All authors. Writing—original draft: W.L., B.B., J.N. Writing-review and editing: All authors. Supervision: W.L., B.B., J.F., W.D. Project administration: W.L., B.B. Funding acquisition: W.L., B.B., J.N., J.G., J.F., W.D.

\setlength{\bibsep}{9.5pt}
\bibliography{sn-bibliography}% common bib file

@article{Atluri2018,
  author  = {Gowtham Atluri and Anuj Karpatne and Vipin Kumar},
  title   = {Spatio-Temporal Data Mining: A Survey of Problems and Methods},
  journal = {ACM Computing Surveys},
  volume  = {51},
  number  = {4},
  pages   = {83:1--83:41},
  year    = {2018},
  doi     = {10.1145/3161602}
}

@inproceedings{niu2026temporal,
  author    = {Niu, Nelson and Osgood, Nathaniel D. and
               Szelko, Jacob S. and Srinivasan, Priyaa Varshinee},
  title     = {Temporal Sheaf Theory for Reconciling Temporal
               Complexity within Public Health Modelling},
  booktitle = {Proceedings of the Ninth International Conference
               on Applied Category Theory},
  year      = {2026},
  note      = {Accepted for publication},
  url       = {https://actconf2026.github.io/papers/ACT\_2026\_paper\_51.pdf}
}

@article{LaxmanSastry2006,
  author    = {Laxman, Srivatsan and Sastry, P. S.},
  title     = {A Survey of Temporal Data Mining},
  journal   = {Sadhana},
  year      = {2006},
  volume    = {31},
  number    = {2},
  pages     = {173--198},
  doi       = {10.1007/BF02719780},
  issn      = {0973-7677},
  url       = {https://doi.org/10.1007/BF02719780}
}

@article{RODDICK1992249,
title = {Temporal semantics in information systems—A survey},
journal = {Information Systems},
volume = {17},
number = {3},
pages = {249-267},
year = {1992},
issn = {0306-4379},
doi = {10.1016/0306-4379(92)90016-G},
url = {https://www.sciencedirect.com/science/article/pii/030643799290016G},
author = {John F Roddick and Jon D Patrick}
}

@article{RoddickSpiliopoulou2002,
  author  = {Roddick, John F. and Spiliopoulou, Myra},
  title   = {A Survey of Temporal Knowledge Discovery Paradigms and Methods},
  journal = {IEEE Transactions on Knowledge and Data Engineering},
  year    = {2002},
  volume  = {14},
  number  = {4},
  pages   = {750--767},
  doi     = {10.1109/TKDE.2002.1019212}
}

@misc{decomposingtimevaryingdata2026,
  author        = {Bumpus, Benjamin Merlin and Nickel, Jana K.},
  title         = {Decomposing time-varying data into simple pieces: structured decompositions of narratives},
  year          = {2026},
  eprint        = {2607.10442},
  archiveprefix = {arXiv},
  primaryclass  = {math.CT},
  doi = {10.48550/arXiv.2607.10442}
}

@book{fundamentalsOfAlgebraicGraphTransformation,
  title     = {Fundamentals of Algebraic Graph Transformation},
  author    = {Ehrig, Hartmut and Ehrig, Karsten and Prange, Ulrike and Taentzer, Gabriele},
  year      = {2006},
  publisher = {Springer Berlin Heidelberg},
  series    = {Monographs in Theoretical Computer Science. An EATCS Series},
  address   = {Berlin, Heidelberg},
  isbn      = {978-3-540-31187-4},
  doi       = {10.1007/3-540-31188-2}
}

@InProceedings{LackAdhesive,
author="Lack, S. and Sobocinski, P.",
editor="Walukiewicz, Igor",
title="Adhesive Categories",
booktitle="Foundations of Software Science and Computation Structures",
year="2004",
publisher="Springer Berlin Heidelberg",
address="Berlin, Heidelberg",
pages="273--288",
isbn="978-3-540-24727-2", 
doi={https://doi.org/10.1007/978-3-540-24727-2\_20}
}

@misc{Bumpus-et-al_Structured-Decompositions,
	title = {Structured Decompositions: Structural and Algorithmic Compositionality}, 
	author = {Benjamin Merlin Bumpus and Zoltan A. Kocsis and Jade Edenstar Master  and Emilio Minichiello},
	year = {2025},
	eprint = {2207.06091v7},
	archivePrefix = {arXiv},
	primaryClass = {math.CT},
	url = {https://arxiv.org/abs/2207.06091v7}, 
}

@article{llanos19,
  author =	 {Eugenio J. Llanos and Wilmer Leal and Duc H. Luu and
                  J{\"u}rgen Jost and Peter F. Stadler and Guillermo
                  Restrepo},
  title =	 {Exploration of the Chemical Space and Its Three
                  Historical Regimes},
  journal =	 {Proceedings of the National Academy of Sciences},
  volume =	 116,
  number =	 26,
  pages =	 {12660-12665},
  year =	 2019,
  doi =		 {10.1073/pnas.1816039116},
  url =		 {https://doi.org/10.1073/pnas.1816039116},
  DATE_ADDED =	 {Thu Sep 30 18:45:43 2021},
}

@article{theorytimevaryingdata2026,
  author       = {Benjamin Merlin Bumpus and Wilmer Leal and James Fairbanks and Martti Karvonen and Fr{\'e}d{\'e}ric Simard},
  title        = {Towards a Unified Theory of Time-Varying Data},
  journal      = {Applied Categorical Structures},
  year         = {2026},
  volume       = {34},
  number       = {3},
  pages        = {23},
  doi          = {10.1007/s10485-026-09860-4},
  url          = {https://doi.org/10.1007/s10485-026-09860-4},
  issn         = {1572-9095}
}

@phdthesis{Ames_Hybrid-Systems,
    author = {Ames, Aaron David},
    title = {A Categorical Theory of Hybrid Systems},
    year = {2006},
    school = {University of California, Berkeley},
    url = {https://www2.eecs.berkeley.edu/Pubs/TechRpts/2006/EECS-2006-165.pdf}
}

@article{Arnborg-Proskurowski,
	title = {Linear time algorithms for {NP}-hard problems restricted to partial $k$-trees},
	author = {Stefan Arnborg and Andrzej Proskurowski},
	year = {1989},
	journal = {Discrete Applied Mathematics},
	volume = {23},
	number = {1},
	pages = {11--24},
	url = {https://doi.org/10.1016/0166-218X(89)90031-0},
}

@article{Bass,
	title = {Covering theory for graphs of groups},
	author = {Hyman Bass},
	year = {1993},
	journal = {Journal of Pure and Applied Algebra},
	number = {89},
	pages = {3--47},
}

@misc{Carmesin-et-al_Graph-decompositions,
	author={Johannes Carmesin and Raphael W. Jacobs and Paul Knappe and Jan Kurkofka},
	title={Canonical graph decompositions and local separations: From infinite coverings to a finite combinatorial theory}, 
	year={2025},
	eprint={2501.16170v1},
	archivePrefix={arXiv},
	primaryClass={math.CO},
	url={https://arxiv.org/abs/2501.16170v1}, 
}

@inproceedings{Casteigts-et-al,
author="Casteigts, Arnaud
and Flocchini, Paola
and Quattrociocchi, Walter
and Santoro, Nicola",
editor="Frey, Hannes
and Li, Xu
and Ruehrup, Stefan",
title="Time-Varying Graphs and Dynamic Networks",
booktitle="Ad-hoc, Mobile, and Wireless Networks",
year="2011",
publisher="Springer Berlin Heidelberg",
address="Berlin, Heidelberg",
pages="346--359",
isbn="978-3-642-22450-8"
}

@article{Courcelle,
	title = {The monadic second-order logic of graphs. I. Recognizable sets of finite graphs},
	author = {Bruno Courcelle},
	year = {1990},
	journal = {Information and Computation},
	volume = {85},
	number = {1},
	pages = {12-75},
	url = {https://www.sciencedirect.com/science/article/pii/089054019090043H},
}

@article{Dechter-Pearl,
	title = {Tree clustering for constraint networks},
	author = {Rina Dechter and Judea Pearl},
	year = {1989},
	journal = {Artificial Intelligence},
	volume = {38},
	number = {3},
	pages = {353-366},
	url = {https://www.sciencedirect.com/science/article/pii/0004370289900374},
}

@misc{Diestel-et-al_Graph-decompositions,
	author={Reinhard Diestel and Raphael W. Jacobs and Paul Knappe and Jan Kurkofka},
	title = {Canonical graph decompositions via coverings}, 
	year={2025},
	eprint={2207.04855v8},
	archivePrefix={arXiv},
	primaryClass={math.CO},
	url={https://arxiv.org/abs/2207.04855v8}, 
}

@article{Fujita_Width,
	author = {Takaaki Fujita},
	title = {A Brief Overview of Applications of Tree-Width and Other Graph
	Width Parameters},
	year = {2025},
	journal = {Applied Mathematics on Science and Engineering},
	volume = {2},
	number = {1},
	pages = {1--20},
}

@article{Halin,
	title = {S-functions for graphs},
	author = {Rudolf Halin},
	year = {1976},
	journal = {Journal of Geometry},
	volume = {8},
	pages = {171--186},
	url = {https://api.semanticscholar.org/CorpusID:120256194},
}

@article{Harary-Gupta,
	author = {F.~Harary and G.~Gupta},
	title = {Dynamic graph models},
	year = {1997},
	journal = {Mathematical and Computer Modelling},
	volume = {25},
	number = {7},
	pages = {79--87},
}

@article{Holme,
	author = {Petter Holme},
	title = {Modern temporal network theory: a colloquium},
	year = {2015},
	journal = {The European Physical Journal B},
	volume = {88},
	number = {234},
}

@article{Holme-Saramaeki,
	author = {Petter Holme and Jari Saramäki},
	title = {Temporal networks},
	year = {2012},
	journal = {Physics Reports},
	volume = {519},
	number = {3},
	pages = {97--125},
}

@misc{Hu,
	title = {Cellular Sheaves on Higher-Dimensional Structures}, 
	author = {Chuan-Shen Hu},
	year = {2025},
	eprint = {2505.23993v3},
	archivePrefix = {arXiv},
	primaryClass = {math.AT},
	url = {https://arxiv.org/abs/2505.23993v3}, 
}

@article{Kempe-et-al,
	author = {David Kempe and Jon Kleinberg and Amit Kumar},
	title = {Connectivity and Inference Problems for Temporal Networks},
	year = {2002},
	journal = {Journal of Computer and System Sciences},
	volume = {64},
	number = {4},
	pages = {820--842},
}

@article{Lauritzen-Spiegelhalter,
	title = {Local Computations with Probabilities on Graphical Structures and Their Application to Expert Systems},
	author = {S.L. Lauritzen and D.J. Spiegelhalter},
	year = {1988},
	journal = {Journal of the Royal Statistical Society: Series B (Methodological)},
	volume = {50},
	number = {2},
	pages = {157--194},
	url = {https://doi.org/10.1111/j.2517-6161.1988.tb01721.x},
}

@article{Liu,
	title = {A Tree Model for Sparse Symmetric Indefinite Matrix Factorization},
	author = {Joseph W.H. Liu},
	year = {1988},
	journal = {SIAM Journal on Matrix Analysis and Applications},
	volume = {9},
	number = {1},
	pages = {26-39},
	url = {https://doi.org/10.1137/0609003},
}

@article{Michail,
	author = {Othon Michail},
	title = {An Introduction to Temporal Graphs: An Algorithmic Perspective},
	year = {2015},
	journal = {Internet Mathematics},
	volume = {12},
}

@article{Robertson-Seymour_Graph-minors-III,
	title = {Graph minors. III. Planar Tree-Width},
	author = {Neil Robertson and P.D Seymour},
	year = {1984},
	journal = {Journal of Combinatorial Theory, Series B},
	volume = {36},
	number = {1},
	pages = {49-64},
	url = {https://www.sciencedirect.com/science/article/pii/0095895684900133},
}

@article{Robertson-Seymour_Graph-minors-XIII,
	author={Neil Robertson and P.D.~Seymour},
	title={Graph minors. XIII. The Disjoint Path Problem},
	year={1995},
	journal={Journal of Combinatorial Theory, Series B},
	volume={63},
	number={1},
	pages={65--110},
}

@article{Robertson-Seymour_Graph-minors-XVII,
	author={Neil Robertson and P.D.~Seymour},
	title={Graph minors. XVII. Taming a Vortex},
	year={1999},
	journal={Journal of Combinatorial Theory, Series B},
	volume={77},
	number={1},
	pages={162--210},
}

@article{Robertson-Seymour-Thomas,
	author={Neil Robertson and Paul Seymour and Robin Thomas},
	title={Sachs' Linkless Embedding Conjecture},
	year={1995},
	journal={Journal of Combinatorial Theory, Series B},
	volume={64},
	number={2},
	pages={185--227},
}

@article{Serre_Arbres,
	title = {Arbres, amalgames, $SL_2$},
	author = {Jean-Pierre Serre},
	year = {1977},
	journal = {Astérisque, Société Mathématique de France, Paris},
	number = {46},
	note = {Rédigé avec la collaboration de Hyman Bass},
	url = {https://www.numdam.org/item/AST_1983__46__1_0.pdf},
}

@book{Serre_Trees,
  title     = {Trees},
  author    = {Serre, Jean-Pierre},
  publisher = {Springer-Verlag},
  address   = {Berlin, Heidelberg},
  year      = {2003},
  series    = {Springer Monographs in Mathematics},
  isbn      = {3-540-44237-5}
}

@inproceedings{Yannakakis,
	title = {Algorithms for Acyclic Database Schemes},
	author = {Mihalis Yannakakis},
	year = {1981},
	booktitle = {Proceedings of the Seventh International Conference on Very Large Data Bases (VLDB '81)},
	address = {Cannes, France},
	publisher = {IEEE Computer Society},
	pages = {82--94},
}

@unpublished{temporal-cellular-sheaves2026,
  author       = {Austin Copeland and Wilmer Leal and Brandon Fallin and Benjamin Merlin Bumpus and James Fairbanks and Warren E. Dixon},
  title        = {A Temporal Cellular Sheaf Framework for Multi-Agent Systems with Switching Topologies},
  year         = {2026},
  note         = {Work in progress},
}

@book{MesbahiEgerstedt2010,
  title     = {Graph Theoretic Methods in Multiagent Networks},
  author    = {Mesbahi, Mehran and Egerstedt, Magnus},
  year      = {2010},
  publisher = {Princeton University Press},
  address   = {Princeton, NJ},
  series    = {Princeton Series in Applied Mathematics},
  isbn      = {978-0-691-14061-2},
}

@ARTICLE{OlfatiSaber2007,
  author={Olfati-Saber, Reza and Fax, J. Alex and Murray, Richard M.},
  journal={Proceedings of the IEEE}, 
  title={Consensus and Cooperation in Networked Multi-Agent Systems}, 
  year={2007},
  volume={95},
  number={1},
  pages={215-233},
  doi={10.1109/JPROC.2006.887293}}

@book{Bullo2009,
  author    = {F. Bullo and J. Cort{\'{e}}s and S. Mart{\'{i}}nez},
  title     = {Distributed Control of Robotic Networks},
  publisher = {Princeton University Press},
  year      = {2009},
  series    = {Princeton Series in Applied Mathematics},
  address   = {Princeton, NJ, USA},
  isbn      = {978-0-691-14195-4}
}

@ARTICLE{Moreau2005,
  author={Moreau, L.},
  journal={IEEE Transactions on Automatic Control}, 
  title={Stability of multiagent systems with time-dependent communication links}, 
  year={2005},
  volume={50},
  number={2},
  pages={169-182},
  doi={10.1109/TAC.2004.841888}}

@inproceedings{BaiGeorgeChakrabortty2020,
  author    = {He Bai and Jemin George and Aranya Chakrabortty},
  title     = {Hierarchical Control of Multi-Agent Systems using Online Reinforcement Learning},
  booktitle = {Proceedings of the 2020 American Control Conference (ACC)},
  pages      = {340--345},
  year       = {2020},
  address    = {Denver, CO, USA},
  publisher  = {IEEE},
  doi        = {10.23919/ACC45564.2020.9147797}
}

@article{SCATTOLINI2009723,
title = {Architectures for distributed and hierarchical Model Predictive Control – A review},
journal = {Journal of Process Control},
volume = {19},
number = {5},
pages = {723-731},
year = {2009},
issn = {0959-1524},
doi = {10.1016/j.jprocont.2009.02.003},
url = {https://www.sciencedirect.com/science/article/pii/S0959152409000353},
author = {Riccardo Scattolini}
}

@article{MELIKER20111,
title = {Spatio-temporal epidemiology: Principles and opportunities},
journal = {Spatial and Spatio-temporal Epidemiology},
volume = {2},
number = {1},
pages = {1-9},
year = {2011},
issn = {1877-5845},
doi = {10.1016/j.sste.2010.10.001},
url = {https://www.sciencedirect.com/science/article/pii/S1877584510000407},
author = {Jaymie R. Meliker and Chantel D. Sloan}
}

@misc{habereder2025,
      title={A Systematic Review of Spatio-Temporal Statistical Models: Theory, Structure, and Applications}, 
      author={Isabella Habereder and Thomas Kneib and Isao Echizen and Timo Spinde},
      year={2025},
      eprint={2511.00422},
      archivePrefix={arXiv},
      primaryClass={stat.AP},
      url={https://arxiv.org/abs/2511.00422}, 
}

@article{Teich2025,
    doi = {10.1371/journal.pcsy.0000058},
    author = {Teich, Marie AND Leal, Wilmer AND Jost, Jürgen},
    journal = {PLOS Complex Systems},
    publisher = {Public Library of Science},
    title = {Diachronic data analysis supports and refines conceptual metaphor theory},
    year = {2025},
    month = {08},
    volume = {2},
    note      = {Article e0000058},
    url = {https://doi.org/10.1371/journal.pcsy.0000058},
    pages       = {e0000058},
    number = {8}
}

@article{CHOI2021105189,
title = {Short-term probabilistic forecasting of meso-scale near-surface urban temperature fields},
journal = {Environmental Modelling \& Software},
volume = {145},
pages = {105189},
year = {2021},
issn = {1364-8152},
doi = {10.1016/j.envsoft.2021.105189},
url = {https://www.sciencedirect.com/science/article/pii/S1364815221002310},
author = {Byeongseong Choi and Mario Bergés and Elie Bou-Zeid and Matteo Pozzi}
}

@Article{Laubichler2013,
author={Laubichler, Manfred D.
and Maienschein, Jane
and Renn, J{\"u}rgen},
title={Computational Perspectives in the History of Science: To the Memory of Peter Damerow},
journal={Isis},
year={2013},
month={03},
publisher={[The University of Chicago Press, The History of Science Society]},
volume={104},
number={1},
pages={119-130},
doi={10.1086/669891},
url={https://doi.org/10.1086/669891}
}

@book{miritello2013temporal,
  title={Temporal patterns of communication in social networks},
  author={Miritello, Giovanna},
  year={2013},
  publisher={Springer International Publishing},
  series={Springer Theses},
  address={Cham, Switzerland},
  isbn={978-3-319-00110-4},
  doi={10.1007/978-3-319-00110-4}
}

@INPROCEEDINGS{9482931,
  author={Zegers, Federico M. and Phillips, Sean and Dixon, Warren E.},
  booktitle={2021 American Control Conference (ACC)}, 
  title={Consensus over Clustered Networks with Asynchronous Inter-Cluster Communication}, 
  year={2021},
  volume={},
  number={},
  pages={4249-4254},
  doi={10.23919/ACC50511.2021.9482931}
  }

@article{johnstone1999note,
  title={A note on discrete {C}onduch{\'e} fibrations},
  author={Johnstone, Peter},
  journal={Theory and Applications of Categories},
  volume={5},
  number={1},
  pages={1--11},
  year={1999}
}

@inproceedings{CurrierLealEtAl2026,
  author    = {Keith Currier and Wilmer Leal and Brandon Fallin and James Fairbanks and Warren E. Dixon},
  title     = {From Local to Global: Sheaf-Theoretic Control Barrier Functions on Manifolds},
  booktitle = {Proceedings of the IEEE 65th Conference on Decision and Control (CDC)},
  year      = {2026},
  publisher = {IEEE},
  address   = {Honolulu},
}

@INPROCEEDINGS{distributed-mas-coordination-2026,
  author={Hanks, Tyler and Riess, Hans and Cohen, Samuel and Gross, Trevor and Hale, Matthew and Fairbanks, James},
  booktitle={2025 IEEE 64th Conference on Decision and Control (CDC)}, 
  title={Distributed Multi-Agent Coordination over Cellular Sheaves}, 
  year={2025},
  volume={},
  number={},
  pages={3057-3064},
  doi={10.1109/CDC57313.2025.11312066}
}

@misc{hanks2025heterogeneousmultiagent,
      title={Heterogeneous Multi-Agent Multi-Target Tracking using Cellular Sheaves}, 
      author={Tyler Hanks and Cristian F. Nino and Joana Bou Barcelo and Austin Copeland and Warren Dixon and James Fairbanks},
      year={2025},
      eprint={2512.24886},
      archivePrefix={arXiv},
      primaryClass={eess.SY},
      url={https://arxiv.org/abs/2512.24886}, 
}

@book{borges1941,
  title     = {El jardín de senderos que se bifurcan},
  author    = {Borges, Jorge Luis},
  year      = {1941},
  publisher = {Editorial Sur},
  address   = {Buenos Aires}
}

\end{document}